\documentclass{article}
\usepackage{amssymb,amsmath,amsthm}

\usepackage{hyperref}
\usepackage{pgfplotstable}
\usepackage{pgfplots}
\usepackage{mathrsfs}

\usetikzlibrary{arrows.meta, shapes.geometric, positioning}
 \usetikzlibrary{cd}
\pgfplotsset{compat=1.15}

\newcommand{\ZZ}{\mathbb{Z}}
\newcommand{\NN}{\mathbb{N}}
\newcommand{\ba}[1]{\begin{array}{#1}}
\newcommand{\ea}{\end{array}}
\newcommand{\dst}{\displaystyle}
\newcommand{\supp}{\mathrm{supp}\,}

\newcommand{\sumset}{\mathrm{Sums}}

\newcommand{\Wp}{\mathit{WP}} 
\newcommand{\notdiv}{\nmid}
\newcommand{\lcm}{\mathrm{lcm}}
\newcommand{\zsat}{{\tt z3}}
\newcommand{\mmod}{\!\!\mod}
\newcommand{\emset}{\varnothing} 

\newcommand{\ab}[1]{\langle #1\rangle}
\newcommand{\akey}{t}  
\renewcommand{\lcm}{\operatorname{lcm}}
\newcommand{\consid}[1]{(({#1}))}  

\newcommand{\MU}[1]{\uplus{#1}}
\newcommand{\Sigg}{\tilde\Sigma}
\newcommand{\Zset}{\mathcal{Z}}
\newcommand{\Stset}[2]{\mathrm{ST}_{#1}(#2)}
\newcommand{\Epic}[1]{\Psi^{#1}}

\newcommand{\cA}{\mathcal{A}}
\newcommand{\cB}{\mathcal{B}}

\newcommand{\myplus}{\boxplus}
\newcommand{\myminus}{\boxminus}

\theoremstyle{plain}
\newtheorem{theorem}{Theorem}[section]
\newtheorem{lemma}[theorem]{Lemma}
\newtheorem{proposition}[theorem]{Proposition}
\newtheorem{corollary}[theorem]{Corollary}

\theoremstyle{definition}
\newtheorem{definition}[theorem]{Definition}
\newtheorem{example}[theorem]{Example}
\newtheorem{remark}[theorem]{Remark}
\newtheorem{conjecture}[theorem]{Conjecture}
\newtheorem{notation}[theorem]{Notation}

\title{The triplication method for constructing strong starters}

\author{Oleg Ogandzhanyants$^\dag$, Sergey Sadov$^\ddag$, Margo Kondratieva$^*$%
}

\begin{document}
	
\maketitle
{\small
$^\dag$  Russian State Pedagogical University, Saint Petersburg, Russia

$^\ddag$  Private school, Moscow, Russia

$^*$ 
Memorial University, St.\ John's NL  A1C~5S7, Canada
}

\begin{abstract}
The triplication method for constructing strong starters  in $\ZZ_{3m}$ 
from starters in $\ZZ_{m}$ 
(say, a starter of order 21 from a starter of order 7) was proposed by the authors in 2025.
The method reduced construction of the particular combinatorial design (a strong starter in a cyclic group)
to solving a Sudoku-type problem 
 -- an independent task with its own tools and techniques available.  The Sudoku-type problem was formulated in terms of the so-called triplication table constructed from a  starter of  order $m$. The method was applicable  for  odd orders $m\ge 7$  not divisible by 3. 
In the present paper, our previous approach is developed in two directions: (1) the definition of the triplication table is generalized, which expands possibilities for its construction to include three base starters or  even ``pseudostarters''; (2) the formulation of the Sudoku-type problem is broadened to embrace various scenarios of ``modular encoding''  and reconstruction of strong starters from its solution. 
A theoretical gain of these developments consists  in the improved understanding of the general structure of the triplication approach. A practical outcome is elimination of  the requirement that 
$m$ be not divisible by 3. This leads to a broader scope of  
 strong starters obtainable by triplication:
 any latent strong starter of odd order $3m$  
 can emerge this way.

 
 \end{abstract}

\section{Introduction}
This paper is a sequel to our paper \cite{OSK25a} which will be referred to often.

The triplication method aims to construct 
a strong starter of order divisible by 3,
$n=3m$. 
We recall properties of strong starters in Section~\ref{ssec:3table}.
Relevant definitions can be also found  in \cite{OSK25a}. 
For more information about starters and related combinatorial designs consult \cite{Dinitz_Handbook_07}. 

In \cite{OSK25a} we described triplication as a multi-part process. First, one 
constructs a {\it triplication table} (TT) given a strong starter of order $m$ not divisible by 3. Then one sets up a  {\it Modular Sudoku Problem} (MSP),
which is a system of certain arithmetical equations and inequalities.
These stages are fully formalized and readily programmed in a computer.
Then, one has to  solve the obtained MSP.  
 This is the hardest part of the process.
Small size cases can be attempted manually, but overall we have resorted to an existing third-party universal solver. 
The  final part, short and simple, serves to recover the final result (a strong starter of order $n$) from the original TT and the 
found solution of MSP  
by means of the Chinese Remainder Theorem (CRT).

An important limitation of the triplication method as given in  \cite{OSK25a} was the requirement that $m$ be not divisible by 3. 
Here we remove this limitation. In addition, we extend the meaning of triplication in two ways. First, in Section~\ref{sec:3ply} we redefine TT; an explicit construction of TT more general than the one given in  \cite{OSK25a} is proposed in  Section~\ref{sec:construcion-tt}.  
Second, in Section~\ref{sec:Sudoku} we outline a general framework to set up an MSP 
and construct a strong starter of order $3m$ based on the given TT and a contingent solution of the MSP.   Theorem~\ref{thm:main} solidifies this framework, ensuring that a TT and a solution of the MSP yield  a strong starter.

\newcommand{\lbracket}{[}
\newcommand{\rbracket}{]}

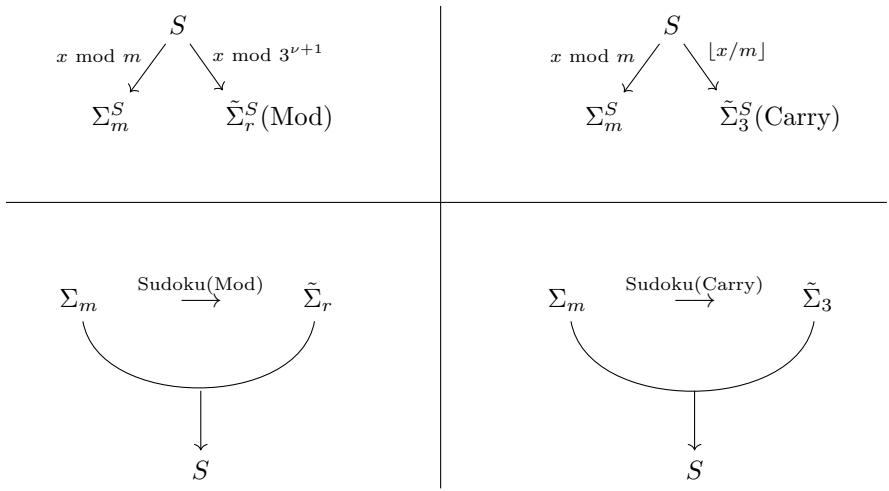
\begin{figure}
{
\newcommand{\tSigma}{\tilde{\Sigma}}
\setlength{\tabcolsep}{10pt}
\renewcommand{\arraystretch}{3}
\begin{tabular}{cc|cc}
 $\dst
 \begin{tikzcd}[column sep=tiny]
 & S\ar[dl, "x\mmod m"'] \ar[dr, "x\mmod 3^{\nu+1}"]
 &
 &[1.5em] \\
 \Sigma_m^S 
 &
& \tSigma_r^S  \text{\lefteqn{\rm(Mod)}}
 &
 & \\[0.5em]
 \end{tikzcd}
 $
 & & &
 %
 %
 $\dst
 \begin{tikzcd}[column sep=tiny]
 & S\ar[dl, "x\mmod m"'] \ar[dr, "\lfloor x/m\rfloor"]
 &
 &[1.5em] \\
 \Sigma_m^S 
 &
& \tSigma_3^S  \text{\lefteqn{\rm(Carry)}}
 &
 & \\[0.5em]
 \end{tikzcd}
 $
%
 %
 \\
 \hline
 $\dst
 \begin{tikzcd}[column sep=tiny]
 &&&&\\[0.5em]
 &\Sigma_m  \ar[rr, bend right=80, dash] 
&
 \stackrel{\rm Sudoku(Mod)}{\longrightarrow} 
  &
\tSigma_r  
 &
\\[0.1em]
 &
& 
 {} \ar[d]
 &
{}
&
\\[0.5em]
&& S
& &
\end{tikzcd}
$
%
& &&
$\dst
 \begin{tikzcd}[column sep=tiny]
 &&&&\\[0.5em]
 &\Sigma_m   \ar[rr, bend right=80, dash] 
&
 \stackrel{\rm Sudoku(Carry)}{\longrightarrow} 
  &
\tSigma_3  
 &
\\[0.1em]
 &
& 
{} \ar[d]
 &
&
\\[0.5em]
&& S
& &
\end{tikzcd}
$
\end{tabular}
}
\label{fig:3schemes}
\caption{Schemes of the triplication method: Mod scenario (left) and Carry scenario (right).
Top row: from strong starter to tables (derivation of conditions and constraints). 
Bottom row: from triplication table $\Sigma_m$ to strong starters.}
\end{figure}

The general framework arose from an effort 
to unify two concrete scenarios of triplication shown in Figure~\ref{fig:3schemes} in the left and right columns respectively. We label them ``Mod" and ``Carry" and discuss their specifics in Sections~\ref{sec:mod} and ~\ref{sec:carry}.
The top row in Figure ~\ref{fig:3schemes} depicts {\em analysis}.
We start with our objective, a strong starter $S$ of order $3m$, and work backward in order to identify properties of the objects corresponding to $S$ involved in the triplication process: the TT  $\Sigma_m^S$
and a solution $\Sigg_r^S$ of the MSP. 
Constraints described in Section~\ref{ssec:sigg-table}  (necessary conditions of MSP's solvability) are based on that analysis.

The bottom row depicts {\em synthesis}\ --- what triplication is designed for. More precisely, it shows 
the flow chart of the algorithm to build a strong starter from a triplication table. The horizontal arrow 
symbolizes solution of an MSP to get a suitable (``congruous'') $\Sigg$-table. Finally, the two tables are combined to obtain a strong starter.

Details of the two scenarios are elaborated upon in Sections~\ref{sec:mod} and \ref{sec:carry}, with examples.

Let $m=3^\nu p$, where $p$ is odd and not divisible by $3$.
In Section~\ref{ssec:3basic} we assume that $\nu=0$ and recall the details of our approach as presented in \cite{OSK25a}. In Section~\ref{sec:9etc} we show
how to set up a Sudoku-type problem modulo $r=3^{\nu+1}$, $\nu\geq 1$, whose solution yields a strong starter of order $3m$. 
This variant of triplication uses 
an extension of CRT where the moduli are not coprime.
The content of Section~\ref{sec:mod} 
was developed in part in the PhD thesis of the first author (OO) defended at Memorial University 
in 2025.

From the computational perspective, solving a Sudoku problem modulo $9$, $27$, etc.\
is more or much more costly than solving a problem modulo $3$.
It is possible to modify the method in such a way as to obtain an MSP modulo 3 regardless
whether 3 is coprime with $m$ or not. This is the content of Section~\ref{sec:carry}. A crucial novelty lies in the
simplest final step: instead of using CRT, we recover numbers from their quotients and remainders of division by $m$
in a straightforward way. 

Due to Theorem~\ref{thm:scenario-independence},
the sets of strong starters obtained from the same triplication table in the two approaches of Section~\ref{sec:mod} 
and~\ref{sec:carry} are the same. 

In the proposed scheme, the construction of a strong starter is contingent on solving an MSP.
In Section~\ref{sec:sudoku-solvability}, we discuss the MSP solvability and give more examples of strong starters constructed by our method. 
In Section~\ref{ssec:ex-sudoku-special} we put forward a conjecture regarding the MSP solvability in the case of a TT constructed by the methods from Section~\ref{sec:construcion-tt}.
In Section~\ref{ssec:ex-sudoku-random}, we discuss more general TTs found numerically.  

Overall, our exposition here does not stress numerical work; there is no systematic report on numerical results,
nor on performance along the lines of Section~8 of \cite{OSK25a}. 
However, all the presented theoretical schemes have been implemented in a Python code and extensively tested. 
The results included in order to illustrate theory constitute a small fraction of all that have been 
computed.

The paper concludes with a brief summary (Section~\ref{sec:conc}). 

It is important in this work to distinguish between the notions
of {\em set}, {\em multiset}, and {\em tuple}. We also use the notion of a {\em dataset} (arrangement for structured data). In Appendix 1 
we review the definitions applicable in our context.

\section{The triplication table}
\label{sec:3ply}


\subsection{Definition of the triplication table $\Sigma_m$} 
\label{ssec:3table}

Let $n=2k+1\geq 3$ be an odd integer. 
Let $S=\{\{x_i, y_i\}, \, i=1,\dots,k\}$ be a strong starter in $\ZZ_{n}=\{0,1,\dots,n-1\}$.  
This means that the (unordered) pairing $S$ has the following properties:

(i) The set of pairs of $S$ is a partition of $\ZZ_{n}^*=\ZZ_{n} \backslash \{0\}$.

(ii) The set of differences $\{\pm(x_i-y_i)_{i=1}^{k}\}$ comprises $\ZZ_{n}^*$.

(iii) Each pair of $S$ makes a unique non-zero sum modulo $n$:

if $\sumset(S)=\{x_i+y_i | \{x_i,y_i\}\in S\}$ then $\sumset(S)\subset \ZZ_{n}^*$ and $|\sumset(S)|=k$. 

(In (ii) and (iii) summation and subtraction of elements in $\ZZ_n$ are done modulo $n$.)

Suppose now that $n=3m$. 
Let us explore the left arrow in the scheme Fig~\ref{fig:3schemes}, top row (either column).

Consider the multiset $S_m$ of pairs of $S$ reduced modulo $m$. 
Each of the  properties of $S$ implies the corresponding property of $S_m$.

Denote by $\MU{S_m}$ the multiset union of all pairs of $S_m$. 

\begin{theorem}
\label{thm:starter-m-reduction}
Let $S$ be a strong starter in $\ZZ_{3m}$, $m=2q+1$, $q\geq 1$.
Let $S_m=\{\{x\mmod m,\, y\mmod m\}\mid \{x,y\}\in S\}$. 
Then

\smallskip
{\rm(i)} 
for $x\in\ZZ_m$, the multiplicity of $x$  in $\MU{S_m}$ equals $3$ if $x\neq 0$ and $2$ if $x=0$;

{\rm(ii)}
$S_m$ contains exactly one pair of type $\{t,\, t\}$, $t\in\ZZ_m^*$ and for every $d=1,\dots,q$
there are exactly three pairs whose differences are $\pm d \pmod m$; 

{\rm(iii)} 
at most three pairs in $S_m$ make the same non-zero sum 
modulo $m$ and at most two pairs in $S_m$ make zero sum modulo $m$;

{\rm(iv)}
No two pairs in $S_m$ are identical. 

\end{theorem}

\begin{proof}
(i)
Since the elements of $\ZZ_{3m}^*$ reduced modulo $m$ comprise a multiset containing exactly three of each elements of $\ZZ_m^*$ and exactly two $0$'s, the same distribution of elements appears in $\MU{S_m}$. 

(ii)  follows from property (ii) of the strong starter $S$. Here $t\ne 0$ because neither of the pairs $\{m,m\}, \{2m,2m\}, \{m,2m\}$ belongs to the strong starter $S$.

(iii) follows from property (iii) of the strong starter $S$. 

(iv) The proof is an adaptation of end of proof of Theorem 7.2 in \cite{OSK25a}.

Suppose there are pairs $\{x,y\}$ and $\{x',y'\}$ in $S$ such that $x\mmod{m}=x'\mmod{m}=u$
and $y\mmod{m}=y'\mmod{m}=v$. We have 
$x=mU+u$, $y=mV+v$, $x'=mU'+u$, $y'=mV'+v$ for some $U,V,U',V'\in\{0,1,2\}$.

By property (ii) of a strong starter, $x-y\not\equiv x'-y'\pmod{3m}$, hence $U-V\not\equiv U'-V'\pmod{3}$.

By property (iii) of a strong starter, $x+y\not\equiv x'+y'\pmod{3m}$, hence $U+V\not\equiv U'+V'\pmod{3}$.

Therefore $U-U'\not\equiv \pm(V-V')\pmod{3}$.
Hence either $U-U'\equiv 0\pmod{3}$ or $V-V'\equiv 0\pmod{3}$.
In either case we obtain a contradiction with
property (i) of a strong starter. For instance, if $U-U'\equiv 0\pmod{3}$, then $x-x'=m(U-U')\equiv 0\pmod{3m}$,
so $x\equiv x'\pmod{3m}$, a contradiction.
\end{proof}

\begin{definition}
\label{def:3table}
 A {\it triplication table} of order $m=2q+1$ 
 is a dataset (pairing) $\Sigma_m=[(u_0,v_0),(u_1,v_1),\dots,(u_{3q},v_{3q})]$ with components from $\ZZ_m$ satisfying the properties 

{\rm(i)} 
an element $x\in\ZZ_m$ is contained in $\Sigma_m$ three times 
if $x\neq 0$ and two times if $x=0$;

{\rm(ii)}
$\Sigma_m$ contains exactly one pair of the type $(t,t)$, namely $u_0=v_0=t$, $t\ne 0$.
For every $d=1,\dots,q$ there are three pairs with directed\footnote{Given an ordered pair $(u,v)$, we define the directed difference as $v-u$. 
} difference $\hat d$ modulo $m$, namely 
\begin{equation}
\label{rowdif3table}
v_{3d-2}-u_{3d-2}=v_{3d-1}-u_{3d-1}=v_{3d}-u_{3d}=\hat d,
\end{equation}
where $\hat d$ is either $d$ or $-d$, but it must be the same for all 3 such pairs.

{\rm(iii)} 
For $s\in\ZZ_m$, there are at most three pairs in $\Sigma_m$ that make the sum $s$ modulo $m$ if $s\neq 0$,
and at most two such pairs if $s=0$. 

{\rm(iv)}
No two pairs in $\Sigma_m$ are identical. 
\end{definition}

The conditions in Definition~\ref{def:3table} essentially repeat the properties of the mod $m$-reduced starter $S$ in Theorem~\ref{thm:starter-m-reduction}, but the pairs are ordered in a special way.
 
 \smallskip
 While mathematically $\Sigma_m$ is defined as an ordered pairing, it is very convenient to display it in the form of a table
using a layout introduced in \cite{OSK25a} according to the following rules.

\begin{itemize}
\item 
 The table contains $q+1$ rows and $3$ columns. The rows are
indexed by $0,1,\dots,q$, the columns are indexed by $0,1,2$. 

\item
The top row, called also the special row, is incomplete. It
 contains a single pair of type $(\akey,\ \akey)$, $\akey\in\ZZ_m^*$.
By convention, we place this pair in column $1$ (central). 
The value $\akey$ is called the {\em key}\ (of the table).
 
\item
The rows indexed $1$ to $q$ are called {\em regular}.
The $i$-th row (for $i=1,\dots, q$) contains pairs 
whose directed differences have common value $\pm i$, cf.~\eqref{rowdif3table}. 

The {\em regular part}\ of the triplication table is comprised by the $3q$ pairs in the $q$ regular rows.
\end{itemize}

\begin{definition}
\label{def:3table-s}
Let $S$ be a strong starter in $\ZZ_{3m}$ and let $S_m$ be its reduction modulo $m$ as in Theorem~\ref{thm:starter-m-reduction}.
Suppose the pairs in $S_m$ and elements in the pairs are ordered%
\footnote{There is a preferred order where $\hat d=d$ for $d=1,\dots,q$, but in general we allow different signs 
of each $\hat{d}$ for compatibility with \cite{OSK25a}.}
 so as to satisfy the conditions $u_0-v_0=0$ and \eqref{rowdif3table}.
The obtained dataset is a triplication table denoted $\Sigma_m^S$. We say that $\Sigma_m^S$
{\em is induced}\ by $S$.
\end{definition}

\begin{remark}
\label{rem:non-unique-sigma-m}
Theorem~\ref{thm:starter-m-reduction} guarantees that
the ordering required in Definition~\ref{def:3table-s} is always possible and the obtained dataset satisfies
the requirements of Definition~\ref{def:3table}. However there is a certain freedom in pair ordering.
Indeed, while the row index for every pair is fixed by Eq.~\eqref{rowdif3table} (the unique pair $(\akey,\akey)$ is placed in the special row), the column index of the pair is not specified. So there are $3!=6$ possibilities to arrange pairs within
every row, thus a total of $6^q$ possible arrangements for $\Sigma_m^S$. Also, there are $2^q$ choices of
signs for the directed differences $\hat d$.
\end{remark}

In view of the latter Remark, we introduce an equivalence relation on the set of triplication tables.

\begin{definition}
\label{def:equiv3tables}
Two TTs of the same order
are {\em equivalent}\ if one can be obtained from the other by (i) some permutation of pairs 
(necessarily within the same row) and (ii) simultaneously swapping all 3 pairs in some rows.
\end{definition}

We denote the equivalence class of a given triplication table $\Sigma_m$ by $[\Sigma_m]$.

Returning to Definition~\ref{def:3table-s} and Remark~\ref{rem:non-unique-sigma-m}, we observe that
to every strong starter $S$ in $\ZZ_{3m}$ there corresponds a unique equivalence class $[\Sigma_m^S]$
of TTs.

\begin{remark}
\label{rem:many-to-one}
If $S$ and $Q$ are two strong starters of order $3m$, the equality
$[\Sigma_m^S]=[\Sigma_m^Q]$ does not imply $S=Q$. 
In general, the correspondence $S\leftrightarrow [\Sigma_m^S]$ is of the type {\em many-to-one}.
\end{remark}

An approach to building a TT has been proposed in \cite{OSK25a}; it will be further developed and generalized
in Section~\ref{sec:construcion-tt}. 
Constructions to be described involve special columnar arrangements of pairs in TTs,
not shared by different tables of the same equivalent class. 

Some terminology related to columnar arrangements will be helpful in the sequel.

\begin{definition}
\label{def:pseudostarter}
A {\em pseudostarter}\ in $\ZZ_m$, $m=2q+1$, is a pairing $\{\{x_i,y_i\},\,i=1,\dots,q\}$
such that $x_i,y_i\in \ZZ_m$ and the differences $\pm(y_i-x_i)\pmod{m}$ comprise $\ZZ_m^*$. 
An {\em ordered pseudostarter}\ is a pairing $[(x_1,y_1),\dots,(x_q,y_q)]$
where $y_i-x_i\equiv \pm i\pmod{m}$, $i=1,\dots,q$.
\end{definition}

Every starter is a pseudostarter, but the opposite is not true. In a pseudostarter, it is not required
that the pairs $\{x_i,y_i\}$ form a partition of $\ZZ_m^*$.  Some elements can be included more
than once and some may be missing. The value $0$ is not forbidden.

\begin{definition}
\label{def:3-balanced}
A triple of pseudostarters $T_1$, $T_2$, $T_3$ in $\ZZ_m$ is {\em balanced}\ 
if the multiset union $T_1\uplus T_2\uplus T_3$ 
contains $0$ exactly twice, one non-zero element once, and every other element of $\ZZ_m^*$ exactly three times. 
\end{definition}

\begin{proposition}
\label{lem:3-balanced}
The three columns of the regular part of a triplication table $\Sigma_m$ 
form a balanced triple of pseudostarters in $\ZZ_m$. 
\end{proposition}

\begin{proof} 
The statement follows from comparison of Definitions \ref{def:3table}, \ref{def:pseudostarter}, and \ref{def:3-balanced}.
The key value $\akey$ is contained twice in the special row. By item (i) of Definition~\ref{def:3table}, 
it must occur exactly once in the regular part of the table.
\end{proof}

\subsection{Objects associated with a triplication table $\Sigma_m$}
\label{ssec:3table-assoc}

Let us introduce some notation and auxiliary objects associated with $\Sigma_m$. 
In the sequel we label the pairs and their entries in $\Sigma_m$ as follows: the pair in the special 0-th row is $(u_0,v_0)$
(so $u_0=v_0=\akey$); the pairs in a regular $k$-th row ($k=1,\dots,q$) are: $(u_{3k-2},v_{3k-2})$, $(u_{3k-1},v_{3k-1})$, $(u_{3k},v_{3k})$.

We will use two-index notation $\ab{i,\ell}$ to indicate positions of the entries (components) in $\Sigma_m$. Here $i\in\{0,\dots,3q\}$ is the index of the cell containing the current entry; $\ell=0$ or $1$: the index $0$ points to $u_i$ and $1$ to $v_i$.  

Residues modulo $m$ are treated as colors.
Monochrome sets are the sets of positions of  entries in the triplication table with equal values (colors).

\begin{definition}
\label{def:cset}
The {\em monochrome set $M_c$ of color $c\in\{0,\dots,m-1\}$}\ is the set of double indices $\ab{i,\ell}$
such that the corresponding entry of the table $\Sigma_m$ has value $c$.
\end{definition}

\begin{proposition}
\label{prop:monoset}
In a triplication table, the sizes of the monochrome sets satisfy:
 $|M_0|=2$ and $|M_c|=3$ if  $c\neq 0$. 
\end{proposition}
\begin{proof}
This is due to item (i) of  Definition~\ref{def:3table}.
\end{proof}

For $s\in \{0,\dots,m-1\}$, 
denote by $W_s$ the set of indices $i$ of pairs $(u_i, v_i)$ in $\Sigma_m$ such that $u_i+v_i\equiv s \!\pmod{m}$.

\begin{proposition}
\label{prop:wset}
In a triplication table, the sizes  of sets $W_s$ satisfy:
 $|W_0|\leq 2$ and $|W_s|\leq 3$ if $s\neq 0$.
\end{proposition}
\begin{proof}
This is due to item (iii) of Definition~\ref{def:3table}.
\end{proof}


\begin{definition}
\label{def:wset}
The set $W_s$ is called a {\em weak set}\ (with sum $s$) if either $s=0$ or $|W_s|>1$.

The sets $\Wp_s=\{(u_i,v_i), i\in W_s\ \text{a weak set}\}$ are called {\em weak pair sets}\ with sum $s$.

Pairs that are not weak are called {\em strong}.
In other words, a pair is strong if its sum is unique and nonzero.
\end{definition}

\section{The Modular Sudoku Problem}
\label{sec:Sudoku}

This section is concerned with tables denoted $\Sigg_r$ and $\Sigg_3$ in Fig.~\ref{fig:3schemes}.

\subsection{Discriminating 3-to-1 modular reduction $\ZZ_{3m}\to \ZZ_m$}
\label{ssec:discr}

The map $x\mapsto x\mmod m$, $\ZZ_{3m}\to\ZZ_m$, is $3$-to-$1$. Hence, knowing $u=x\pmod{m}$,
in order to recover $x\in\ZZ_{3m}$, we need a way to make a choice between 3 candidate values of $x$.
This can be done uniquely if to every $x\in\ZZ_{3m}$ we put in correspondence not just $u$ but a pair $(u,U)$,
where $u=x\pmod{m}$ and $U$ is a {\em discriminator} --- a numerical parameter that  takes
different values for $x=u$, $x=u+m$, and $x=u+2m$.
 
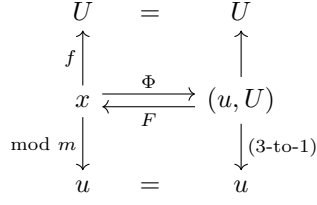
\begin{figure}
$$
 \begin{tikzcd}[column sep=tiny]
U 
 &[0.3em]=&[0.3em]
 U
 \\
 x\ar[d, "\mmod m"'] \ar[rr, shift left, "\Phi"] \ar[u, swap, "f"']
 &&
 (u,U)\ar[d, swap, "\text{(3-to-1)}"'] \ar[ll, shift left, "F"] \ar[u, "{}"']
  \\
u 
 &[0.3em]=&[0.3em]
 u
 \end{tikzcd}
$$
\caption{A discrimination scenario}
\label{diagram:discr-scenario}
\end{figure}
 
Let us encode a number $x\in \ZZ_{3m}$ by a pair 
$
(u,U)\in\ZZ_m\times\Zset,
$
where 
\begin{equation}
\label{u-proj}
u=x\mmod{m}\quad  \text{and} \quad U=f(x), 
\end{equation}
$\Zset$ is a finite set  of cardinality $\geq 3$,
and $f:\ZZ_{3m}\to\Zset$ is the {\em discriminating function}, such that
the map $\Phi:\,x\mapsto(u,U)$ is one-to-one.

Let $\Omega\subset \ZZ_m\times\Zset$ be the range of the map $\Phi$.
Since $\Phi$ is one-to-one, there exists the inverse (decoding) map $F:\Omega\to \ZZ_{3m}$, which recovers $x$ given $(u,U)$. 

This abstract scheme, visualized in Fig.~\ref{diagram:discr-scenario}, will be called a {\em discrimination scenario} (or just
``scenario'' for short). 

For reference, let us summarize the discriminating property of the correspondence $x\leftrightarrow (u,U)$ in a
 formal proposition, which is in fact a tautology.
  
 \begin{proposition}
 \label{prop:xy}
 Let $x, y \in \ZZ_{3m}$ and  $\Phi(x)=(u,U)$, $\Phi(y)=(u,V)$, where $\Phi: \ZZ_{3m}\to\ZZ_m\times \Zset$
 is one-to-one. Then $x\ne y$ iff $U\ne V$.
\end{proposition}

In practice, we will use two particular scenarios,
which we will label as ``Mod" and ``Carry" and 
subscript the functions $f$, $\Phi$, and $F$ accordingly.

\begin{itemize}
\item
{\em Mod scenario}: assuming $m=3^\nu p$, $\nu\ge 0$, and $3\notdiv p$,
\begin{equation}
  \label{mod-sc}
f_{\rm Mod}(x)= x \mmod {3^{\nu+1}}.
\end{equation}
Here $\Zset=\ZZ_{3^{\nu+1}}$, $\Omega=\left\{(u,U)\in\ZZ_m\times\Zset\,:\; 
U\in\{u,\,u+3^{\nu},\,u+2\cdot 3^{\nu}\}\right\}$
(with addition modulo $3^{\nu+1}$).
 
 \item
{\em Carry scenario}: regardless of divisibility of $m$ by $3$,
   \begin{equation}
  \label{carry-sc}
  f_{\rm Carry}(x)=\lfloor x/m\rfloor.
  \end{equation}
Here $\Zset=\ZZ_3$ and $\Omega=\ZZ_m\times\Zset$.  
 \end{itemize} 
  
  In either scenario, given $x\in \ZZ_{3m}$, the pair $(u,U)$  (\ref{u-proj}) is found by one of the formulas \eqref{mod-sc} or \eqref{carry-sc}. 
Further details, in particular, the inverse maps $F:(u,U)\to x$, are treated in Sections~\ref{sec:mod}
and \ref{sec:carry} (where the nickname ``Carry'' is clarified), respectively.
  
 \begin{example}
Table~\ref{table:mod-discr} illustrates discrimination in the Mod scenario. It displays, for each 
 $u\in \ZZ_m$, the corresponding values of $x\in \ZZ_{3m}$ and the discriminators $U=f_{\rm Mod}(x) $,
 see Eq.~\eqref{mod-sc},
 in the cases $m=7=3^0\cdot 7$ (left) and $m=15=3^1\cdot 5$ (right). 
 
In the Carry scenario, each cell in the ``$U$'' column of both tables,
calculated according to Eq.~\eqref{carry-sc}, would consist of values $0,1,2$ in that order. 

 \begin{table}
 \centerline{
 \begin{tabular}{|c|c|c|}
  \hline 
 $u$ & $x\in\ZZ_{21}$ & $U$\\
 \hline 
 0& 0, 7, 14 & 0, 1, 2\\
 \hline 
 1& 1, 8, 15 & 1, 2, 0\\
  \hline 
 2& 2, 9, 16 & 2, 0 ,1\\
 \hline 
 3& 3, 10, 17 & 0, 1 ,2\\
\hline 
 4& 4, 11, 18 & 1, 2, 0\\
  \hline 
 5& 5, 12, 19 & 2, 0 ,1\\
  \hline 
 6& 6, 13, 20 & 0, 1, 2\\
  \hline 
 \end{tabular}
 \hspace{3em}
 \begin{tabular}{|c|c|c|}
  \hline 
 $u$ & $x\in\ZZ_{45}$ & $U$\\
 \hline 
 0& 0, 15, 30 & 0, 6, 3\\
 \hline 
 1& 1, 16, 31 & 1, 7, 4\\
  \hline 
 2& 2, 17, 32 & 2, 8, 5\\
 \hline 
 3& 3, 18, 33 & 3, 0, 6\\
  \hline 
 4& 4, 19, 34 & 4, 1 ,7\\
 \hline 
 \dots &\dots& \dots\\
 \hline 
 14& 14, 29, 44 & 5, 2, 8\\
  \hline 
 \end{tabular}
 } 
 \caption{Correspondences $x\leftrightarrow (u,U)$ in Mod scenario for $m=7$ (left) and $m=15$ (right)}
 \label{table:mod-discr}
 \end{table}
 
  \end{example} 
 
 
\subsection{``Arithmetic'' on discriminators}  
\label{ssec:squareops} 
 
 In the general setting described in Sec.~\ref{ssec:discr}
let us introduce symbolic notation for  ``difference'' $U \myminus V$ 
 and ``sum" $U \myplus V$  in $\Zset$ induced from the difference and sum of two elements  in $\ZZ_{3m}$ as follows:
 $$
 f(x)\myminus f(y)=f( x-y), 
\qquad
 f(x)\myplus f(y)=f( x+y),
 $$
 where $+$ and $-$ in the right-hand sides are the usual arithmetic operations modulo $3m$.

Suppose $(u,U)$ and $(v,V)$ are in $\Omega$. The above equations can be written as
 \begin{equation} 
\begin{array}{rcl}
\label{myplusmin}
 U\myminus V&=&f( F(u,U)-F(v,V)), \\[1ex]
 U\myplus V&=&f( F(u,U)+F(v,V)).
 \end{array}
 \end{equation} 
These formulas may, but in general do not, define
two-argument functions from a Cartesian product of a certain set (say, $\Zset$) into itself. This is why we put the word
``arithmetic'' in the heading of this subsection in quotation marks.  
Rigorous interpretation of the formulas \eqref{myplusmin} is this:
{\em the left-hand sides are understood just as abbreviations of the right-hand sides}.
In general, every time the ``operations''  $\myplus$, $\myminus$ are encountered, the values $u$ and $v$ are to be given along with $U$ and $V$. With this understanding, as we will see in the next subsection,  this notation will be very helpful. 
 
 In the Mod scenario, with $U=f_{\rm Mod} (x)$, $V=f_{\rm Mod} (y)$, formulas (\ref{myplusmin}) become   \begin{equation}
 \label{omop}
 U\myminus V= (U-V) \mmod {3^{\nu+1}}\quad  {\rm and}\quad U \myplus V= (U+V) \mmod {3^{\nu+1}}.
 \end{equation}
 They agree with usual arithmetic operations modulo $3^{\nu+1}$. In this case we may think of $\myplus$, $\myminus$
 as of the correctly defined standard binary operations in $\Zset=\ZZ_{3^{\nu+1}}$.
 
 In the Carry scenario, the situation is more complicated: in order to calculate $U \myminus V$ and $U \myplus V$ by (\ref{myplusmin})  
 we must know their counterparts $u$ and $v$ in the pairs that belong to $\Omega$.
 The explicit formulas are derived in Section~\ref{sec:carry}, see Eq.~\eqref{omop1}. 

\smallskip
The following proposition is an extension of Proposition~\ref{prop:xy}. 
It will be used to discriminate 
results of arithmetical operations modulo $3m$ if the results modulo $m$ coincide.

 \begin{proposition}
 \label{prop:xyz}
 Suppose, for $i=1,2$, $x_i, y_i \in \ZZ_{3m}$,  $\Phi(x_i)=(u_i,U_i)$, and $\Phi(y_i)=(v_i,V_i)$.
 
 {\rm (a)} 
 Suppose that $u_1+v_1\equiv u_2+v_2 \pmod{m}$.
Then
 $x_1+y_1\ne x_2+y_2$  in $\ZZ_{3m}$ iff  $\ U_1\myplus V_1\ne U_2\myplus V_2$ in $\Zset$.
 
 {\rm (b)} 
  Suppose that $u_1-v_1\equiv u_2-v_2 \pmod{m}$.
Then
$x_1-y_1\ne x_2-y_2$ in $\ZZ_{3m}$ iff $\ U_1\myminus V_1\ne U_2\myminus V_2$ in $\Zset$.
 
 \end{proposition}
 \begin{proof}
(a)  Introduce $z_i=x_i+y_i$, $w_i=u_i+v_i$, $W_i=U_i\myplus V_i$, for $i=1,2$. 
 The claim becomes:
 ``Let $w_1=w_2$. Then $z_1\ne z_2$ iff $W_1\ne W_2$''. Since $\Phi(z_i)= (w_i, W_i)$ by (\ref{myplusmin}), the claim is true by Proposition~\ref{prop:xy}.
 
 (b) We use the same argument as in  (a) with $z_i=x_i-y_i$, $w_i=u_i-v_i$, $W_i=U_i\myminus V_i$, for $i=1,2$. 
 \end{proof}
 
\subsection{A table $\Sigg^S$ 
consisting of discriminating values}
\label{ssec:sigg-table}
  
In Sec.~\ref{ssec:3table} we explained the left arrow in the schemes shown in Fig~\ref{fig:3schemes} of the top row and either column.
Here we will similarly explain the right arrow and its target, the table $\Sigg^S_*$.

Let $S$ be a strong starter of order $3m$, $m=2q+1$.  
In an abstract discrimination scenario, along with the table $\Sigma^S_m=\{(u_i,v_i)\}_{i=0}^{3q}$, let us consider a table $\Sigg^S=\{(U_i,V_i)\}_{i=0}^{3q}$ of the same layout and  filled with ordered pairs of elements of $\Zset$,
the target set of the function $f$ in \eqref{u-proj}.

Thus the pairs $(U,V)$ of $\Sigg_{}^S$ 
obtained 
from the pairs $(x,y)$ of a strong starter $S$ 
are aligned with pairs $(u,v)$ of $\Sigma_m^S$ in the following way  (for $0\le i\le 3q$):
\begin{equation}
\label{tables-from-S} 
  \ba{lll}
 (u_i, U_i)&=\;(x_i\!\mmod m,\; f(x_i)) &=\; \Phi(x_i),
 \\[0.5ex]
 (v_i, V_i)&=\;(y_i\!\mmod m,\; f(y_i)) &=\; \Phi(y_i).
 \ea
\end{equation}
This fixed alignment makes the values  $U_i\myminus V_i$, $U_i\myplus V_i$ 
unambiguously defined by Eqs.~\eqref{myplusmin} with $u=u_i$, $v=v_i$.
 
Keeping the description still abstract, but having our two concrete scenarios in mind, we identify $\Zset$
with number set $\ZZ_r=\{0,1,\dots,r-1\}$. In the Mod scenario with $m=3^{\nu}p$, $3\notdiv p$,
 we have $r=3^{\nu+1}$, while in the Carry scenario $r=3$ for any $m$.

The additional condition, where $0_a$ denotes zero element of $\ZZ_a$, 
\begin{equation}
\label{F00}
f(0_{3m})=0_r
\end{equation}
will be henceforth assumed. 
It implies that $\Phi(0_{3m})=(0_m,0_r)$ and $F(0_m,0_r)=0_{3m}$. 
In particular, the condition \eqref{F00} is fulfilled in the scenarios Mod and Carry.
 
As the set $\Zset$ is agreed upon, from now on we will write $\Sigg_{r}^S$ instead of $\Sigg^S$. 
By construction, the table $\Sigg_{r}^S$ has the following properties,
or, as we prefer to say, obeys the following constraints.
 (Double parentheses are used to distinguish between constraints and equation numbers in references.)
 
\begin{enumerate}
\item[\consid{0}]
{\em Range Constraints}.
 Variables $U_i, V_i$ take values in the set $\{0,1, \dots, r-1\}$.

 \item[\consid{1}]
{\em Row Constraints}.
The ``differences" $D_i=U_i\myminus V_i$  in each regular row of the table are distinct. 
The ``difference" $D_0=U_0\myminus V_0$ 
is nonzero (i.e. $D_0 \in \{1,\dots, r-1\}$). 

\item[\consid{2}]
{\em Weak Set Constraints}.

Recall that weak sets associated with table $\Sigma_m^S$ are defined in Definition~\ref{def:wset}.

For every weak set $W_s$:

\begin{enumerate}
\item[(a)] In the case $s\neq 0$, 
the ``sums"  
$\{U_i\myplus V_i,\, i\in W_s\}$ are distinct. 

\item[(b)] In the case $s=0$, 
the ``sums" $\{U_i\myplus V_i,\ i\in W_0\}$ are distinct and nonzero. 
 
\end{enumerate}

\item[\consid{3}]
{\em Color Constraints}.

Recall that monochrome sets $M_c$ associated with table $\Sigma_m^S$ are defined in Definition~\ref{def:cset}.

\begin{enumerate}

\item[(a)] In the case $c\neq 0$, 
the values $U_i$ or $V_i$ in the positions
corresponding to the double indices $\ab{i,\ell}\in M_c$ are distinct. 

\item[(b)] In the case $c=0$, 
the values $U_i$ or $V_i$ in the positions
corresponding to the double indices $\ab{i,\ell}\in M_c$ are distinct and nonzero.

\end{enumerate}

\item[\consid{4}] 
{\em Consistency Constraints}.
The pairs $(u_i,U_i)$ and $(v_i,V_i)$ lie in the range $\Omega$ of the map $\Phi$ defining the
discrimination scenario currently in use.

\end{enumerate}

The constraints ((0))--((3)) directly correspond to the constraints 1--4 in \cite[Sec.~6]{OSK25a},
where, in our present terminology, the Mod scenario was used in the case $\nu=0$ (i.e. $m$ coprime with $3$).
In that case, the constraints ((4)) are vacuous. They will also be vacuous in the Carry scenario.

In the Mod scenario with $\nu>0$
the constraints ((4)) take the form of {\em Modular Compatibility
Conditions}:
$U_i\equiv u_i\!\!\pmod{3^{\nu}}$ and $V_i\equiv v_i\!\!\pmod{3^{\nu}}$ for $i\in\{0,1,\dots,3q\}$,
cf.~the description of the set $\Omega$ next to Eq.~\ref{mod-sc}.

\smallskip
Since the map $\Phi$ is invertible,
and due to the Consistency Constraints, the strong starter $S$ used to construct the tables $\Sigma_m^S$ and $\Sigg^S_{r}$ can be uniquely recovered from these two tables: $x_i=F(u_i,U_i)$, $y_i=F(v_i,V_i)$, $i\in\{0,1,\dots,3q\}$.

\subsection{Congruous table, Modular Sudoku Problem, and the recovery of a strong starter} 
\label{ssec:sigma-to-starter}

While the departure point of the preceding discussion had been a strong starter $S$ in $\ZZ_{3m}$,
the properties (constraints) ((0))--((4)) did not mention $S$ explicitly; implicitly $S$ was still present via the TT  $\Sigma^S_m$. 
Now we want to completely eliminate $S$ from the picture.
The new point of departure will be a triplication table {\em per se}, regardless of its origin.
 
 \begin{definition} 
\label{def:mod-Sudoku}
Let there be given:

(i) a triplication table $\Sigma_m$ of order $m$, see Definition~\ref{def:3table},  

(ii) a discrimination scenario defined by a set $\Zset=\ZZ_r$ 
and a function $f:\,\ZZ_{3m}\to \Zset$
such that the map $\Phi:\,x\mapsto(u,U)$ defined by Eq.~\eqref{u-proj} is one-to-one.

Let $\{W_s\}$ and $\{M_c\}$ be the collections of weak, resp., monochrome sets associated with $\Sigma_m$
as defined in Sec.~\ref{ssec:3table-assoc} and operations $\myplus$, $\myminus$ be defined by 
\eqref{myplusmin} in the chosen scenario.

\smallskip
A table $\Sigg_r$ satisfying the constraints {\rm ((0))--((4))} stated above is said to be
{\em congruous with}\ TT $\Sigma_m$ (in the given scenario).

\smallskip
The following problem is called the {\em Modular Sudoku Problem}\
(related to $\Sigma_m$ in the given scenario):

\smallskip
{\em Find a table $\Sigg_r$  congruous with $\Sigma_m$.}
\end{definition}

 \begin{definition} 
\label{def:mod-Sudoku-starter}
Suppose $\Sigma_m$ is a 
TT
of order $m=2q+1$ and $\Sigg_r$ is a table congruous with $\Sigma_m$
in the given discrimination scenario with mutually inverse maps $\Phi$ and $F$. 
For aligned pairs $(u_i,v_i)$ of $\Sigma_m$ and $(U_i,V_i)$ of $\Sigg_r$ put 
\begin{equation}
\label{recover-starter}
x_i=F(u_i,U_i), 
\quad
y_i=F(v_i,V_i)
\quad\text{for $i=0,\dots, 3q$}.
\end{equation}
The right-hand sides are defined due to the Consistency Constraints. 

The so obtained pairing $[(x_i,y_i)]_{i=0}^{3q}$ is said to be {\em congruous with}\ 
TT
$\Sigma_m$
in the given scenario. To add specificity, the pairing $[(x_i,y_i)]_{i=0}^{3q}$ is said to be {\em congruous with $\Sigma_m$  via}\ 
$\, \Sigg_r$. Referring to Eqs.~\eqref{recover-starter}, we also say that  $\Sigma_m$ and $\Sigg_r$ {\em yield}\
the pairing $S$.
\end{definition}

The analysis in Sec.~\ref{ssec:sigg-table} can be summarized in the following formal statement.

\begin{proposition}
\label{prop:constraints-necessary}
Suppose $S$ is a strong starter of order $3m$,
$\Sigma_m^S$ is the triplication table induced by $S$ (see Definition~\ref{def:3table-s}),
and $\Sigg_r^S$ is the table defined in Sec.~\ref{ssec:sigg-table} in 
the given discrimination scenario (with $\Zset=\ZZ_r$).
Then

{\rm(i)} the table $\Sigg_r^S$ is congruous with $\Sigma_m^S$;

{\rm(ii)} the original starter $S$ is congruous with $\Sigma_m^S$ via $\Sigg_r^S$.
\end{proposition}

In this sense, the constraints ((0))--((4)) are necessary conditions for the tables 
$\Sigma_m$ and $\Sigg_r$
to yield a strong starter by means of Eqs.~\eqref{recover-starter}.

We do not claim that these conditions are sufficient for the existence of a $\Sigg_r$ congruous with $\Sigma_m$;
this is not even true --- see Section~\ref{sec:sudoku-solvability}.
But the next weaker statement already justifies the triplication idea.


\begin{theorem}
\label{thm:main}
Let $\Sigma_m=[(u_i,v_i)]_{i=0}^{3q}$ be a 
triplication table of order $m=2q+1$ (as described in Definition~\ref{def:3table}).
Fix a discrimination scenario with set $\Zset=\ZZ_r$, $r\ge 3$, 
and decoding function $F$.
Suppose $\Sigg_r$ is a table congruous with $\Sigma_m$ in the given scenario.
Then the pairing $S=[(x_i,y_i)]_{i=0}^{3q}$
defined by Eqs.~\eqref{recover-starter},
that is, the unique pairing congruous with $\Sigma_m$ via $\Sigg_r$,
is a strong starter%
\footnote{The resulting strong starter will be written as a set of {\em unordered}\ pairs $\{\{x_i,y_i\}\}$, while in the process of construction we always work with {\em ordered}\ pairing written as $[(x_i,y_i)]$.}
in $\ZZ_{3m}$.
\end{theorem}

\begin{proof}
Since $F(0_m,0_r)=0_{3m}$ as follows from  assumption \eqref{F00},
the constraint ((3)) by Proposition~\ref{prop:xy} implies that $S$ is a partition of $\ZZ_{3m}^*$.
The constraint ((1)) by Proposition~\ref{prop:xyz} (b) implies that $S$ is a starter in $\ZZ_{3m}$. 
The constraint ((2)) by Proposition~\ref{prop:xyz} (a) implies that  $S$ is strong.
\end{proof}

We emphasize that the existence of a table $\Sigg_r$ congruous with given 
TT
$\Sigma_m$ in the given
scenario is not guaranteed. To find $\Sigg_r$ amounts to solving an MSP and we offer no
theoretical recipe for that.
 
 \subsection{Scenario-independence of the set of congruous starters}
\label{ssec:mod-carry}

We mentioned two concrete discrimination scenarios and there can conceivably be many more.
The natural question arises: 
{\em is any scenario better than others, either universally or for the particular given triplication table $\Sigma_m$?}
Let us understand ``better'' as ``yielding more solutions''. In particular, can it happen that for the same $\Sigma_m$
the Modular Sudoku Problem would have no solution in one scenario but at least one solution in another scenario?

The answer is: no.  In all valid scenarios, the MSP
and the recovery formulas \eqref{recover-starter}
supply the same set of strong starters.
 
\begin{notation}
\label{def:sudoku-solution-set}
Fix $m=2q+1$. Let $\cA$ be a discrimination scenario satisfying assumption \eqref{F00}.
For the given triplication table $\Sigma_m$, the set of (ordered) strong starters congruous with $\Sigma_m$
in scenario $\mathcal{A}$ is denoted $\Stset{\cA}{\Sigma_m}$.
\end{notation}  
  
 \begin{theorem}
\label{thm:scenario-independence}
Fix $m=2q+1$.
Let $\Sigma_m$ be a triplication table as defined in Section \ref{sec:3ply}.
Suppose $\mathcal{A}$ and $\mathcal{B}$ are two discrimination scenarios satisfying assumption \eqref{F00}.
Then $\Stset{\cA}{\Sigma_m}=\Stset{\cB}{\Sigma_m}$.
\end{theorem}

\begin{proof}
The roles of $\mathcal{A}$ and $\mathcal{B}$ are symmetric, so it suffices to prove that
$\Stset{\cA}{\Sigma_m}\subset\Stset{\cB}{\Sigma_m}$.

Suppose $S\in \Stset{\cA}{\Sigma_m}$.
Then $S$ is a strong starter and $\Sigma_m$ is induced by $S$ (see Definition~\ref{def:3table-s}):
$\Sigma_m=\Sigma_m^S$.

In scenario $\cB$, let $\Zset=\ZZ_r$  and $\Phi_{\cB}$ be the encoding function. 
Consider the table $\Sigg_r^S$ with components $U_i$, $V_i$
constructed from $S$ by means of Eqs.~\eqref{tables-from-S} where $\Phi=\Phi_{\cB}$.

By Proposition~\ref{prop:constraints-necessary}, 
$\Sigg_r^S$ is congruous with $\Sigma_m^S$ in scenario $\cB$ and 
$S$ is congruous with $\Sigma_m^S$ in scenario $\mathcal{B}$ via $\Sigg_r^S$.
Therefore $S\in \Stset{\cB}{\Sigma_m}$.
\end{proof}

Due to Theorem~\ref{thm:scenario-independence}, we can drop the subscript in the notation of the set of strong
starters obtainable by triplication from the given 
TT
and write $\Stset{}{\Sigma_m}$.

For ordered starters from $\Stset{}{\Sigma_m}$ the set of corresponding unordered starters depends only on the equivalence class $[\Sigma_m]$.

\begin{center}
*\,*\,*
\end{center}

Our practical goal is generation of strong starters $S$ in $\ZZ_{3m}$, so we will now turn to 
concrete implementation of the proposed abstract theory. 
In the next section we describe 
some explicit constructions of the triplication table $\Sigma_m$. 
Once
$\Sigma_m$ is 
generated, 
the construction of $S$ proceeds in three steps:

\smallskip  
 I. 
 Setting up the Modular Sudoku Problem.

\smallskip  
 II. 
 Solving the Modular Sudoku Problem, that is, finding $\Sigg_{r}$.

\smallskip  
 III. 
 Recovering the strong starter $S$ from $\Sigma_m$ and $\Sigg_{r}$.

\smallskip  
We provide details of steps I and III in the Mod and Carry scenarios in Sections \ref{sec:mod}  and \ref{sec:carry}, respectively. Specifically, we describe the inverse (decoding) map $F$ in both scenarios and 
reveal the operations $\myminus$ and $\myplus$ in the Carry scenario in a fully explicit form. 



\section{Explicit constructions of the triplication table}
\label{sec:construcion-tt}

In this section we explore some explicit constructions of the triplication table $\Sigma_m$ formally described by  Definition \ref{def:3table}. The construction  described in Sec.~\ref{ssec:basic-3} was proposed in the PhD thesis of Ogandzhanyants
and presented in \cite{OSK25a}. The idea of Sec.~\ref{ssec:epicycloidal} was inspired by Horton's paper \cite{Horton_71}.
Altogether, these led to the unified construction described in Sec.~\ref{ssec:3-template}.

\subsection {Triplication template, admissible keys}
\label{ssec:3-template}

According to Proposition~\ref{lem:3-balanced}, the columns of a triplication table must comprise a balanced
triple of pseudostarters. 
We start with specializing the kind of balanced triples that will be used in subsequent constructions.

\begin{definition}
\label{def:2-balanced}
A pair of pseudostarters (Def.~\ref{def:pseudostarter}) 
$T_1$, $T_2$ in $\ZZ_m$ is 
{\em special}\ 
if the multiset union $T_1\uplus T_2$ (taking multiplicities of elements into account) 
contains every element of $\ZZ_m^*$ exactly twice.
\end{definition}

Suppose that there are given:

\begin{itemize}
\item
An ordered starter $T_0$ 
and a special pair of  ordered pseudostarters  $T_1, T_2$ 
of order $m=2q+1$.
Let
$$
 T_j=[(x_i^{(j)},y_i^{(j)})]_{i=1}^{q},\quad j=0,1,2.
$$
Moreover, we require 
$T_0$, $T_1$ and $T_2$ 
to be ordered {\em consistently}, so that the sign in 
\begin{equation}
\label{rowdifs}
y_i^{(j)}-x_i^{(j)}=\pm i,
\end{equation}
is the same for $j=0,1,2$. (Cf.~Definition~\ref{def:pseudostarter}.)

\item
An integer $\akey$ modulo $m$ (a {\em key}),
$\akey\in\{1,\dots,m-1\}$. 
\end{itemize}

\begin{definition}
\label{def:3-template}
The {\em triplication template}\ 
$\Sigma^0_m(T_0,T_1,T_2,\akey)$
is a dataset of the same structure as the triplication table introduced in Section~\ref{ssec:3table}:
the special row  with a single pair $(\akey, \akey)$
and $q$ regular rows of length $3$ each. Column $0$ is a copy of the starter $T_0$, while the 
components of the regular rows in columns $1$, $2$ depend on the components of
the corresponding pseudostarter and the key $t$:
\begin{equation}
\begin{array}{rcl}
\label{3table-gen}
( u_{3i-2}, v_{3i-2})&=&(x_i^{(0)}, y_i^{(0)}), \\[0.5ex]
( u_{3i-1}, v_{3i-1})&=&(\akey+x_i^{(1)}, \akey+y_i^{(1)}), \\[0.5ex]
( u_{3i}, v_{3i})&=&(\akey+x_i^{(2)}, \akey+y_i^{(2)}).
\end{array}
\end{equation}
Here $1\leq i\leq q$.
\end{definition}

We can regard a template 
as a {preliminary} triplication table because  it satisfies conditions (i)---(iii) of Definition~\ref{def:3table}.
It may or may not satisfy condition (iv). 
The condition (i) is satisfied because $\akey\neq 0$ and the pair $(T_1,T_2)$ is special.

\begin{definition}
\label{def:key-admissible}
For a fixed triple 
$(T_0, T_1, T_2)$  (a starter and a special pair of pseudostarters),
the set of {\em admissible key values}\ is 
$$
K(T_0,T_1,T_2)=\{\akey \in \ZZ_m^* \mid\text{the template $\Sigma^0_m(T_0,T_1,T_2,t)$ contains no equal pairs}\}.
$$
\end{definition}

Thus,  a template with an admissible value of key is a triplication table
$$
\Sigma_m(T_0,T_1,T_2,\akey)=\Sigma^0_m(T_0,T_1,T_2,\akey), 
\quad \akey\in K(T_0,T_1,T_2).
$$

 \begin{definition}
\label{def:disjoint}
Two pseudostarters that have no pair in common are called {\it disjoint}. 
 \end{definition}

 \begin{remark}
\label{rem:sym}
1. Due to the symmetry between the 2nd and 3rd formulas in 
(\ref{3table-gen}), $K(T_0,T_1,T_2)=K(T_0,T_2,T_1)$ and the tables
$\Sigma_m(T_0,T_1,T_2,\akey)$ and  $\Sigma_m(T_0,T_2,T_1,\akey)$ are equivalent (see Definition~\ref{def:equiv3tables}). 

\smallskip
2. Clearly, if pseudostarters $T_1$ and $T_2$ are not disjoint, then $K(T_0,T_1,T_2)=\emset$.
\end{remark}

Below we consider three special  cases of  construction~\eqref{3table-gen}. 

\subsection{Triplication based on one starter of order $m$}
\label{ssec:basic-3}

For any ordered pairing $T=[(x_i,y_i)]_{i=1}^{I}$ with components in $\ZZ_m$, the conjugate pairing%
\footnote {By analogy with unordered case, where two sets $T$ and $T'$  of unordered pairs are called conjugate if
$\{x,y\}\in T\,\Leftrightarrow\,\{-x,-y\}\in T'$, cf.~\cite[Definition~3.2]{OKSh24}.}
 is defined as
\begin{equation}
\label{conj}
T'=[(-y\!\mmod m,\; -x\!\mmod m)]_{i=1}^{I}.
\end{equation}

Let $T=[(x_i,y_i)]_{i=1}^q$ be an ordered starter in $\ZZ_m$, $m=2q+1$. 
Assign $T_0=T_1=T$ and $T_2=T'$. 
The resulting template $\Sigma^0_m(T,T,T',t)$ depends only on one starter, $T$, and a key $t$. 
In this case we will  write $\Sigma^0_m(T,t)$ as a shortcut for $\Sigma^0_m(T,T,T',t)$
and  $K(T)$ as a shortcut for $K(T,T,T')$.
This is the simplest setup, which was introduced in \cite{OSK25a}.

\begin{example}
\label{ex:caseI}
Let  $m= 7$ and $T=[(2, 3), (4, 6), (5, 1)]$. 

Then $T'=[(7-3,7-2), (7-6,7-4), (7-1,7-5)]=[(4,5), (1,3), (6,2)]$.

Here is the template $\Sigma^0_7(T,3)$ for the key value $\akey=3$:

\bigskip

\centerline{
	\begin{tabular}{|c||c||c|}
		\hline 
		&$(3, 3)$ &   \\
		\hline
		 $(2, 3)$  &  $(3+2, 3+3)$&  $(3+4, 3+5)$\\ 
		\hline
		  $(4, 6)$ &  $(3+4, 3+6)$&  $(3+1,3+3)$ \\ 
		\hline
		 $(5,  1)$ &  $(3+5, 3+1)$&  $(3+6, 3+2)$ \\ 
		\hline
	\end{tabular} $\!\pmod 7$\, =\,
	\begin{tabular}{|c||c||c|}
		\hline 
		&$(3, 3)$ &   \\
		\hline
		 $(2, 3)$  &  $(5, 6)$&  $(0, 1)$\\ 
		\hline
		  $(4, 6)$ &  $(0, 2)$&  $(4,6)$ \\ 
		\hline
		 $(5,  1)$ &  $(1, 4)$&  $(2, 5)$ \\ 
		\hline
	\end{tabular} 
}

\bigskip

This template has no identical pairs, so $3\in K(T)$ is an admissible key value and
$\Sigma_7(T,3)=\Sigma^0_7(T,3)$ is a triplication table.
\end{example}

The set of admissible keys $K(T)$
is easy to characterize. Here is the summary of our results from  \cite{OSK25a} (Lemmas 2,\ 3,\ 4 and Theorem~2). 

\begin{proposition}
\label{prop:num-good-keys}
Let $T$ be a starter in $\ZZ_m$. Consider the set of pair sums
$$
\sumset(T)=\{x+y \mmod m\mid(x,y)\in T\}.
$$
If $0\in\sumset(T)$, then $K(T)=\emset$;
otherwise,
$K(T)=\ZZ_m^*
\setminus \sumset(T)$.
\end{proposition}

\begin{corollary}
\label{cor:keys-case1}
(a) If $T=\{(x,m-x)\mid 1\leq x\leq (m-1)/2\}$ is the {canonical (patterned) starter}\ of order $m$,
then $K(T)=\emset$.

\smallskip
(b) If $T$ is a strong starter of order $m$, then $|K(T)|=(m-1)/2$.
\end{corollary}

\subsection{Triplication based on three starters of order $m$}
\label{ssec:3base_starters}

In this case, $T_0$, $T_1$ and $T_2\ne T_1$ are 
any consistently
ordered starters in $\ZZ_m$. 

On the one hand, this is a generalization of the previous construction ($T_1=T_0$, $T_2=T'_0$).
On the other hand, this is a specialization of construction \eqref{3table-gen}: here $T_1$ and $T_2$ are starters,
not just pseudostarters.

 \begin{remark}
Let us mention some nuances related to Remark~\ref{rem:sym} and Corollary~\ref{cor:keys-case1}. 

\begin{enumerate}
\item
Consider the following triple of ordered strong starters in $\ZZ_{13}$:
$$
\ba{l}
R = [ (3,4),  (6,8),  (9,12),    (10,1),   (2,7),   (5,11) ],
\\[0.3ex]
S= [(3,4), ( 5,7 ), (   9,12  ), (10,1 ),  (  6,11), (    2,8  )],
\\[0.3ex]
T= [(9,10), (    5,7  ), (  1,  4  ), (    12,3 ),   (     6,11),(    2,8  )].
\ea
$$

Here $S$ has common pairs with both $R$ and $T$, hence
$K(R, S, T)=K(T, R, S)=\emset$.
However, $K(S,R,T)=\{1,2,3,5,6,9\}$. 

\item

It is possible that three pairwise disjoint strong starters have empty set of admissible keys, see 
Example~\ref{ex:ord11-all_disjoint}, Eq.~(\ref{empty}).  

\item
In general, the number of admissible key values can vary, see 
examples in Section 7. 
We do not know any generalization of Corollary~\ref{cor:keys-case1}(b).
\end{enumerate}

\end{remark}

\subsection{Triplication with epicycloidal pseudostarters}
\label{ssec:epicycloidal}

In this version of construction \eqref{3table-gen} we take an arbitrary starter $T_0$ and 
a pair of pseudostarters $T_1$ and $T_2$ obtained by the following special construction.

\begin{definition}
\label{ps-st}
Let $m=2q+1$ and $\mu\in\{2,\dots,m-2\}$ be such that $\mu-1$ 
is coprime with $m$. An {\em epicycloidal pseudostarter}\ of order $m$ with {\em multiplier}\ $\mu$ has the form
$$
 \Epic{m}(\mu)=[(x_i(\mu), \mu x_i(\mu)), \; i=1,\dots,q],
$$
where $x_i(\mu)$ is the solution of the equation $(\mu-1)x_i(\mu)= i$ in $\ZZ_m$.
\end{definition}

\begin{remark}
The adjective {\it epicycloidal} refers to the geometric analogy: the envelope of chords $[\theta,\mu \theta \pmod{2\pi}]$ in the unit circle, where $\mu$ is an integer $\geq 2$, is an epicycloid with $\mu-1$ cusps. The epicycloid with one cusp 
is called the cardioid. The adjective {\em cardioidal} was proposed in \cite{OKSh18} for starters in cyclic groups consisting of pairs $(x,2x)$, developing the idea of starters' visualization found in the PhD thesis of Dinitz \cite{Dinitz-Stinson_81}.
The epicycloidal generalization was mentioned in \cite[p.~6]{OKSh18} but left without elaboration.

By analogy with term ``$k$ quotient starter'' introduced in \cite{Dinitz_81}, one could say that $\Epic{m}(\mu)$
is a ``$1$-quotient pseudostarter''. In our opinion, though, the word ``quotient'' is overloaded in this
context and,  as we deal with cyclic groups only,  the allusion to geometry is a safe way to resolve potential terminological collisions.
\end{remark}

Let $\Epic{(m)'}(\mu)$ be a pseudostarter conjugate to $\Epic{(m)}(\mu)$ as defined in (\ref{conj}).
\begin{proposition}
\label{prop:epi-pair} 
{\rm (a)} Let $\mu-1$ and $\mu$ be coprime with $m$.
Then the pseudostarters $\Epic{(m)}(\mu)$ and $\Epic{(m)'}(\mu)$ form a 
special pair (see Definition~\ref{def:2-balanced}).

{\rm (b)} If, in addition, $\mu+1$ is coprime with $m$ then $\Epic{(m)}(\mu)$ and $\Epic{(m)'}(\mu)$  are disjoint
 (see Definition~\ref{def:disjoint}).
\end{proposition}

\begin{proof}
(a) The set of components $x_i(\mu)$ in $\Epic{(m)}(\mu)$ is the image of the set $D=\{1,2,\dots,\frac{m-1}2\}$
under the map $\phi:\; d\mapsto (\mu-1)^{-1} d\mmod m$. 
The set of components $m-x_i(\mu)$ in $\Epic{(m)'}(\mu)$ is the image of $\ZZ_m^*\setminus D$
under the same map. 
The union of the two sets is $\ZZ_m^*$, since $\phi$ is a bijection on $\ZZ_m^*$.

Similarly, the set comprising all $y_i=\mu x_i(\mu)$ in $\Epic{(m)}(\mu)$ and all $m-y_i$ in $\Epic{(m)'}(\mu)$
is $\ZZ_m^*$; the bijection in this case is $\psi: \; d\mapsto \mu(\mu-1)^{-1} d$.

(b) Invertibility of $\mu+1$ ensures that $x_i\neq m-y_i$ for all $i$. 
(Conversely, if $\gcd(\mu+1,m)>1$, the pseudostarters $\Epic{(\mu)}$ and $\Epic{(\mu)'}$ do have identical pairs.)

\end{proof}

\begin{corollary}
\label{cor:case3}
{\rm (a)} If $m$ is coprime with $\mu$, $\mu-1$ and $\mu+1$, then $(\Epic{(\mu)}, \Epic{(\mu)'})$
is a special pair of disjoint pseudostarters. 

{\rm (b)} 
There are no special pairs of disjoint epicycloidal pseudostarters if $3|m$. 
\end{corollary}

\begin{example}
\label{psev-tt}
Let  $m= 7$,
$T_0=[(2, 3), (4, 6), (5, 1)]$,
and $\mu=2$. Then
$$
 T_1=\Epic{(7)}(2)=[(1,2), (2,4), (3,6)], \qquad
  T_2=\Epic{(7)'}(2)=[(5,6), (3,5), (1,4)].
$$ 
Here is the template $\Sigma^0_7(T_0,\,T_1,\,T_2,\;3)$ for the key  $\akey=3$:

\bigskip

\centerline{
	\begin{tabular}{|c||c||c|}
		\hline 
		&$(3, 3)$ &   \\
		\hline
		 $(2, 3)$  &  $(4,5)$&  $(1, 2)$\\ 
		\hline
		  $(4, 6)$ &  $(5, 0)$&  $(6,1)$ \\ 
		\hline
		 $(5,  1)$ &  $(6, 2)$&  $(4, 0)$ \\ 
		\hline
	\end{tabular}
}
\smallskip

This template is a triplication table  $\Sigma_7(T_0,\,T_1,\,T_2,\;3)$ and it will be used in Example~\ref{ex:pseu} for construction of a strong starter of order 21.

\end{example}

\begin{center}
*\,*\,*
\end{center}


Practical methods for obtaining triplication tables are not limited to the outlined constructions.
In particular, it is possible to generate triplication tables treating them just as combinatorial designs
satisfying Definition~\ref{def:3table} and using algorithms of randomized search. Some examples will be presented
in Section~\ref{ssec:ex-sudoku-random}. 

In the next two sections we will describe construction of a strong starter of order $3m$
under two particular discrimination scenarios, assuming that we already have a triplication table $\Sigma_m$ found by any method and satisfying  Definition~\ref{def:3table}. 
They are the Mod and Carry scenarios mentioned in Section~\ref{sec:Sudoku}.

\section{Triplication with Mod discrimination}
\label{sec:mod}
\subsection{Obtaining starters of order $3m$ with $m$ coprime to  3}
\label{ssec:3basic}

The method for constructing strong starters of order $3m$, where
 $m=2q+1$ is odd and not divisible by 3, is described in \cite{OSK25a}. 
 The following is a short summary. 

\begin{enumerate}

\item[Step I.]  For a triplication table $\Sigma_m $, set up
a Modular Sudoku Problem,
which is a system of linear equations and inequalities in $\ZZ_3$, to be solved for $3m-1$ unknowns. Specifically, find the weak sets and monochrome sets and impose the constraints ((0)--((3)) with operations $\myminus$, $\myplus$ defined by (\ref{omop}).
Here $\nu=0$ (thus $r=3$).  The constraints ((4)) of Section~\ref{ssec:sigg-table},
which in the Mod scenario take the form of Modular Compatibility Conditions, are trivial in the case $\nu=0$. 

This part has low computational complexity $O(m)$. 

\item[Step II.]
Solve the obtained MSP.
(A solution, if it exists, is not unique.) 
This is an algorithmically complex part. It may use various methods and techniques. In \cite{OSK25a}
the MSP was solved using a universal SAT solver {\zsat}.  
\end{enumerate}

\begin{enumerate}
\item[ Step III.]
This part, conceptually simple and algorithmically fast ($O(\log m)$ per entry),
builds the final answer ---
a strong starter  of order $3m$ --- from the found solution of the MSP.


We have two tables of pairs of values: $\Sigma_m$ (the triplication table),
where the values belong to $\{0,1,\dots,m-1\}$, and $\Sigg_3$, now filled with
numbers in $\{0,1,2\}$ satisfying the imposed constraints.

Each table contains $3q+1$ pairs. To the pairs in the same position taken from the two tables,
$(u_i, v_i)$ and $(U_i,V_i)$ we put in correspondence the pair $(x_i,y_i)$ such that
$x_i,y_i\in \{0,1,\dots,3m-1\}$, $x_i\equiv u_i\pmod m$, $x_i\equiv U_i\pmod 3$,
and similarly for $y_i$.
The existence and uniqueness of $(x_i,y_i)$ (given $u_i, v_i, U_i, V_i$) follows from CRT.
\end{enumerate}

In the language of abstract algebra, we recover an element of the abelian group $\ZZ_{3m}\cong
\ZZ_m\oplus\ZZ_3$ (by means of CRT)
from its projections onto the direct summands:

\begin{equation}
\label{diagram:CRTprojections}
 \begin{tikzcd}[column sep=tiny]
 & \ZZ_{3m}\ar[dl, dashed, "\pi_m"'] \ar[dr, dashed, "\pi_3"]
 &
 &[1.5em] \\
 \ZZ_m
 &
& \ZZ_3
 &
 &
 \end{tikzcd}
\end{equation}
Here $\pi_m$ and $\pi_3$ are the reductions $x\mapsto x\pmod{m}$ and $x\mapsto x\pmod{3}$, respectively. 
In Section~\ref{sec:Sudoku}, the map $\Phi_{\rm Mod}=(\pi_m,\pi_3)$ was described  by 
Eqs.~\eqref{u-proj}, \eqref{mod-sc} with  $\nu=0$.

\smallskip

 Theorem~1 in \cite{OSK25a} asserts that if the constraints listed 
 in Step I
 are fulfilled then $x_i\neq 0\neq y_i$ for $i\in\{0,1,\dots,3q\}$,  
the pairs $(x_i,y_i)$ constitute 
a strong starter in $\ZZ_{3m}$. This is a special case of Theorem~\ref{thm:main} with operations $\myminus$, $\myplus$ defined by (\ref{omop})
with $\nu=0$ and $F_{\rm Mod}$ realized via CRT.

\subsection{Obtaining starters of order $3m$ with $m$ divisible by  3}
\label{sec:9etc}

Let now $m=3^\nu p$, where $\nu\geq 1$ and $3\notdiv p$.
In this case the triplication method as formulated in Section~\ref{ssec:3basic} is not applicable, because
CRT cannot be used to recover values from the residues modulo $m$ and modulo $3$.
We will describe a suitable modification of the method.

In Step I, given a triplication table $\Sigma_m$, we impose the constraints ((0)-((4)) with operations $\myminus$, $\myplus$ defined by (\ref{omop}) and $\nu\ge 1$.
A significant difference between this section and Section~\ref{ssec:3basic} is that a 
Sudoku problem is formulated modulo $r=3^{\nu+1}$ rather than modulo $3$.
In Step II we obtain $\Sigg_{3^{\nu+1}}$, provided that the solution exists. 
In Step III of the triplication process we  recover the components of a new starter of order $3m$
from their residues modulo $m$ and modulo $3^{\nu+1}$.
(For clarity we treated the case $\nu=0$ separately in Section~\ref{ssec:3basic}, although 
technically it could be included here as a special case.)

Since $\gcd(m, 3^{\nu+1})=3^{\nu}$, for the system of congruences
\begin{equation}
\label{congmod3m}
 x\equiv u\mmod m, \qquad
 x\equiv U\mmod 3^{\nu+1}
\end{equation}
to have a solution, $u$ and $U$ must satisfy the Modular Compatibility Condition 
\begin{equation}
\label{compat}
u\equiv U\mmod 3^\nu,
\end{equation}
which represents, in this scenario, 
a constraint of type ((4)) of Section~\ref{ssec:sigg-table}.

From a more abstract point of view, at Step III we want to recover an element of an abelian group
from two its projections onto factorgroups in the presence of compatibility condition:

\begin{equation}
\label{diagram:projections}
 \begin{tikzcd}[column sep=tiny]
 & \ZZ_{3m}\ar[dl, dashed, "\mmod m"'] \ar[dr, dashed, "\mmod 3^{\nu+1}"]
 &
 &[1.5em] \\
 \ZZ_m\ar[dr, "\mmod 3^\nu"'] 
 &
& \ZZ_{3^{\nu+1}}\ar[dl, "\mmod 3^\nu"]
\\
 & \ZZ_{3^\nu} 
 &
 &
 \end{tikzcd}
\end{equation}

\bigskip
This parallels the view of CRT as a manifestation
of the isomorphism $\ZZ_{3m}\simeq\ZZ_m\oplus\ZZ_3$ in the case $\gcd(m,3)=1$.
Indeed, 
Diagram~\eqref{diagram:projections} makes sense also if $\nu=0$, in which case the bottom vertex
of the rhombus is the trivial group $\{0\}$, the compatibility condition is a tautology, and \eqref{diagram:projections} 
reduces to Diagram~\eqref{diagram:CRTprojections}.

\smallskip
 
As shown in Appendix 2, 
the following method can be used to recover the solution $x$ of equations \eqref{congmod3m} under condition \eqref{compat}.

Let $\bar u=u\pmod {3^\nu}$. Then $(u-\bar u)/3^\nu$ and $(U-\bar u)/3^{\nu}$ are integers.
Since gcd$(p,3)=1$, one can use CRT to find $ x'$ such that
\begin{equation}
\label{CRT-gen-1}
 x'\equiv\frac {u-\bar u}{3^\nu} \!\!\pmod p, \qquad  x'\equiv \frac {U-\bar u}{3^\nu} \!\!\pmod 3.
\end{equation}
The answer is given by
\begin{equation}
\label{CRT-gen-2}
x=(\bar u+ 3^\nu x') \!\mmod{3m}\,=\,F_{\rm Mod}(u,U). 
\end{equation}
The so defined function  $F_{\rm Mod}$ 
is used in (\ref{recover-starter}) to recover a strong starter in the Mod scenario 
for $\nu\ge1$. 

For convenience we give an explicit formulation of the abstract Theorem~\ref{thm:main} adapted to the
Mod scenario. At the same time, it
is an extension for the case $\nu\ge 1$ of Theorem 1 in \cite{OSK25a}.

\begin{proposition}
\label{prop:main-mod}
Let $\Sigma_m$ be a triplication table, where $m=3^\nu p$, $\nu\ge 1$,  $3\notdiv p$. 
Suppose, in Step I of the triplication process, the Sudoku problem is set up using the operations $\myminus$, $\myplus$ given by (\ref{omop}).
Suppose in Step II a table $\Sigg_{3^{\nu+1}}$ congruous with  $\Sigma_m$ is found. 
If the pairing $S$  is constructed by formulas (\ref{recover-starter}) with $F_{\rm Mod}$ given by (\ref{CRT-gen-1}), (\ref{CRT-gen-2}) and  $(u_i,v_i) \in \Sigma_m$, 
$(U_i,V_i) \in \Sigg_{3^{\nu+1}}$, $0\le i\le \frac{3(m-1)}2$, then $S$ is a strong starter in $\ZZ_{3m}$. 
\end{proposition}

\begin{example}
\label{ex:sudoku45}
Let $m=15=3^1\cdot 5$, so $\nu=1$. We start with  triplication table $\Sigma_{15}=\Sigma_{15}(T,t)$ constructed by formulas (\ref{3table-gen}) in the one-starter based setup (Sec.~\ref{ssec:basic-3})
using key $\akey=4$ and the ordered starter 
$$
T=[(3,4),(12,14),(7,10),(2,6),(8,13),(5,11),(9,1)].
$$
 Our goal is to find a strong starter of order 45 by solving an MSP modulo $9$.

\smallskip
The tables $\Sigma_{15}$ (numerical) and $\Sigg_9$ (with indeterminates) are as follows:

\bigskip
\begin{tabular}{|c||c||c|}
\hline
&   4, 4  & \\
\hline
3, 4 & 7, 8 & 0, 1 \\
\hline
12, 14 & 1, 3 & 5, 7 \\
\hline
7, 10 & 11, 14 & 9, 12 \\
\hline
2, 6 & 6, 10 & 13, 2 \\
\hline
8, 13 & 12, 2 & 6, 11 \\
\hline
5, 11 & 9, 0 & 8, 14 \\
\hline
9, 1 & 13, 5 & 3, 10 
\\
\hline
\end{tabular}
\hspace{2em}
\begin{tabular}{|rl||rl||rl|}
\hline
&  & $U_0$, & $V_0$ & & \\
\hline
$U_1$, & $V_1$ & $U_2$, & $V_2$ & $U_3$, & $V_3$ \\
\hline
$U_4$, & $V_4$ & $U_5$, & $V_5$ & $U_6$, & $V_6$ \\
\hline
$U_7$, & $V_7$ & $U_8$, & $V_8$ & $U_9$, & $V_9$ \\
\hline
$U_{10}$, & $V_{10}$ & $U_{11}$, & $V_{11}$ & $U_{12}$, & $V_{12}$ \\
\hline
$U_{13}$, & $V_{13}$ & $U_{14}$, & $V_{14}$ & $U_{15}$, & $V_{15}$ \\
\hline
$U_{16}$, & $V_{16}$ & $U_{17}$, & $V_{17}$ & $U_{18}$, & $V_{18}$ \\
\hline
$U_{19}$, & $V_{19}$ & $U_{20}$, & $V_{20}$ & $U_{21}$, & $V_{21}$
\\
\hline
\end{tabular}

\bigskip

The weak sets and the weak pair sets (see Definition~\ref{def:wset}) are
$$
\begin{array}{lcl}
 W_0=\{2,12\} & \to & \Wp_0=\{(7,8),(13,2)\},\\
 W_1=\{3,11,16\} & \to &\Wp_1=\{(0,1),(6,10),(5,11)\},\\
 W_2=\{7,15\} & \to &\Wp_2=\{(7,10),(6,11)\},\\
 W_6=\{9,13\} & \to &\Wp_6=\{(9,12),(8,13)\},\\
 W_7=\{1,18\} & \to &\Wp_7=\{(3,4),(8,14)\},\\
 W_8=\{0,10\} & \to &\Wp_8=\{(4,4),(2,6)\},\\
 W_{10}=\{8,19\} & \to &\Wp_{10}=\{(11,14),(9,1)\}.
 \end{array}
$$

The monochrome sets mapped to the subsets of indeterminates are
$$
\begin{array}{lclcl}
 M_0=\{U_3, V_{17}\} &&
 M_5=\{U_6,U_{16},V_{20}\} &&
 M_{10}=\{V_7,V_{11},V_{21}\}\\
 M_1=\{V_3,U_5,V_{19}\} &&
 M_6=\{V_{10},U_{11},U_{15}\} &&
 M_{11}=\{U_8,V_{15},V_{16}\}\\
 M_2=\{U_{10},V_{12},V_{14}\} &&
 M_7=\{U_2,V_6,U_7\}&&
 M_{12}=\{U_4,V_9,U_{14}\}\\
 M_3=\{U_1,V_5,U_{21}\}&&
 M_8=\{V_2,U_{13},U_{18}\}&&
 M_{13}=\{U_{12},V_{13},U_{20}\}\\
 M_4=\{U_0,V_0,V_1\}&&
 M_9=\{U_9,U_{17},U_{19}\}&&
 M_{14}=\{V_4,V_8,V_{18}\}\\
 \end{array}
$$

The Modular Compatibility Conditions stipulate that the values in table $\Sigg_9$ must have the following
residues modulo $3$: 

\bigskip
\centerline{
\begin{tabular}{|c||c||c|}
\hline
&   1, 1  & \\
\hline
0, 1 & 1, 2 & 0, 1 \\
\hline
0, 2 & 1, 0 & 2, 1 \\
\hline
1, 1 & 2, 2 & 0, 0 \\
\hline
2, 0 & 0, 1 & 1, 2 \\
\hline
2, 1 & 0, 2 & 0, 2 \\
\hline
2, 2 & 0, 0 & 2, 2 \\
\hline
0, 1 & 1, 2 & 0, 1 
\\
\hline
\end{tabular}
}

\bigskip
This data constitute a complete set-up  (Step I) for the MSP modulo 9.

Table~\ref{tab:sudoku45} (left) is a solution (Step II) for $\Sigg_9$ found by hand 
before a computer program was written.
Table~\ref{tab:sudoku45} (right) is the combination  (Step III) of the tables $\Sigma_{15}$ and $\Sigg_9$ as the solution of the projection problem
\eqref{diagram:projections} using \eqref{CRT-gen-1}--\eqref{CRT-gen-2} with $\nu=1$.

\begin{table}
\centerline{
\begin{tabular}{|c||c||c|}
\hline
& 1,  7 &  \\
\hline
3,  4 & 1, 5 & 6, 4 \\
\hline
3, 5 & 1, 0 & 8, 4 \\
\hline
7, 1 & 5, 2 & 6, 6 \\
\hline
8, 3 & 0, 7 & 1, 2 \\
\hline
8, 7 & 0, 5 & 6, 8 \\
\hline
2, 2 & 0, 3 & 2, 8 \\
\hline
3,7 & 4, 5 & 6,  4
\\
\hline
\end{tabular}
\hspace{2em}
\begin{tabular}{|c||c||c|}
\hline
& 19, 34& \\
\hline
3, 4 & 37, 23 & 15, 31 \\
\hline
12, 14 & 1,  18 & 35, 22 \\
\hline
7, 10 & 41,  29 & 24, 42 \\
\hline
17,  21 & 36,  25 & 28, 2 \\
\hline
8,  43 & 27, 32 & 6,  26 \\
\hline
20,  11 & 9,  30 & 38,  44 \\
\hline
39,  16 & 13, 5 & 33,  40 
\\
\hline
\end{tabular}
}
\caption{Construction of a strong starter of order
${45}$ (Example~\ref{ex:sudoku45})}
\label{tab:sudoku45}
\end{table}

The constructed starter is
    $S=\{\{19, 34\},  \{3, 4\}, \dots, \{33,  40\}\}$.
\end{example}

One may wonder whether there is something special about the case $p=1$, where the order of a starter
to build is a power of 3. The answer is no: it's all the same. 

\begin{example}
\label{ex:m9allkeys}
In this example we use  formulas \eqref{3table-gen} and the setup of Sec.~\ref{ssec:basic-3} with starter 
$$
T=[(5,6),(2,4),(7,1),(8,3)]
$$
of order $m=9$.
We have constructed the triplication tables $\Sigma_9(T, t)$ for all admissible keys 
$\akey\in K=\{1,3,4,5,7\}$. 
For every $\akey\in K$ the obtained 
Sudoku problem modulo 27
has a solution.
In this case a strong starter $S$ of order 27 can be read off the table $\Sigg_{27}$ directly.
    Sample solutions are listed in  Table~\ref{table:SS-27}.

\begin{table}
\centerline{
\begin{tabular}{|c|p{9.4cm}|}
\hline
Key & 
{\hfil{Sample starter}\hfil}
\\
\hline
1 & $\{\{19, 1\}, \{23, 15\}, \{6, 16\}, \{4, 5\}, \{20, 13\}, \{3, 14\}, \{24, 26\}, $
\newline $\{10, 25\}, \{2, 8\}, \{21, 18\}, \{12, 17\}, \{22, 9\}, \{11, 7\}\}$\\
\hline
3 & $\{\{21, 12\}, \{5, 6\}, \{17, 9\}, \{15, 25\}, \{20, 4\}, \{14, 7\}, \{8, 10\},$
\newline $ \{1, 16\}, \{22, 19\}, \{23, 2\}, \{3, 26\}, \{24, 11\}, \{13, 18\}\}$ \\
\hline
4 &  $\{\{4, 13\}, \{14, 24\}, \{18, 19\}, \{7, 26\}, \{2, 22\}, \{6, 8\}, \{9, 20\},$
\newline $ \{1, 25\}, \{23, 11\}, \{15, 21\}, \{3, 17\}, \{16, 12\}, \{5, 10\}\}$\\
\hline
5 & $\{\{23, 14\}, \{5, 24\}, \{10, 20\}, \{8, 9\}, \{2, 22\}, \{7, 18\}, \{1, 3\}, $
\newline $ \{19, 16\}, \{6, 12\}, \{25, 13\}, \{21, 26\}, \{17, 4\}, \{15, 11\}\}$\\
\hline
7 &  $\{\{7, 25\}, \{14, 24\}, \{12, 4\}, \{10, 11\}, \{20, 13\}, \{18, 2\}, \{21, 23\}, $
\newline $ \{19, 16\}, \{17, 5\}, \{9, 15\}, \{3, 26\}, \{1, 6\}, \{8, 22\}\}$\\
\hline
\end{tabular}
}
\caption{Sample strong starters of order 27 for all admissible keys (Example~\ref{ex:m9allkeys})}
 \label{table:SS-27}
 \end{table}

\end{example}

\section{Triplication with Carry discrimination}
\label{sec:carry}

Let us look carefully at the freedom of selection of candidate values for the unknowns $U_i$, $V_i$ in 
the Sudoku problem modulo $3^{\nu+1}$ described in Section~\ref{sec:9etc}. Due to the 
Modular Compatibility Conditions --- a special kind of constraints \consid{4}, --- every $U_i$ can assume not $3^{\nu+1}$ values but just three: $u_i$, $u_i+3^{\nu}\pmod{3^{\nu+1}}$
and $u_i+2\cdot 3^{\nu}\pmod{3^{\nu+1}}$.
Hence, the Sudoku problem of Section~\ref{sec:9etc} is to be solved, in essence, for a set of $3m-1$ unknowns
with values in the set $\{0,1,2\}$, just like in the basic variant ($\nu=0$) treated in Section~\ref{ssec:3basic}.

This fact motivates us to seek a variant of the method that explicitly involves unknowns defined modulo 3.
Such a variant is presented in this section.

In this variant,  at step III of the triplication process we recover the resulting starter
of order $3m$ from two sets of data (modulo $m$ and modulo $3$) by an elementary school arithmetic formula: 
given $u\in\{0,1,\dots,m-1\}$ and $U\in\{0,1,2\}$, let
\begin{equation}
\label{carry-recover}
 x=F_{\rm Carry}(u,U) =(mU+u) \mmod{3m}.
\end{equation}

That is, $x$ is recovered modulo $3m$ from its quotient $U$ and remainder $u$ of division by $m$.
Eq.~\eqref{carry-recover} defines the inverse map to the map $\Phi_{\rm Carry}$ defined by Eqs.~\eqref{u-proj} and \eqref{carry-sc} in Sec.~\ref{ssec:discr}.

The correspondence
$ \ZZ_m\times \ZZ_3\ni (u,U)\;\leftrightarrow\; x\in \ZZ_{3m}$
is a bijection of {sets}.

This recovery recipe is applicable regardless of divisibility of $m$ by $3$ and
is easier than CRT-based recovery in Section~\ref{sec:mod}.
The computational complexity is also lower: $O(1)$ per entry as opposed to $O(\log m)$ in the Mod scenario.

The price to pay is the loss of homomorphic nature of the operations $\myplus$, $\myminus$.

In the Mod scenario, given the pairs $(u_1, U_1)$ and $(u_2, U_2)$ and the corresponding values $x_1$, $x_2$ such that  
$x_i\equiv u_i\pmod{m}$, $x_i\equiv U_i\pmod{3}$ ($i=1,2$), we know that the pair $(u_1+u_2, U_1+U_2)$
corresponds to $x_3=x_1+x_2\mmod {3m}$, that is, $x_3\equiv u_1+u_2\pmod{m}$, $x_3\equiv U_1+U_2\pmod{3}$.

Let us see what happens when formula \eqref{carry-recover} is used.

Suppose $x_i=mU_i+u_i$, $i=1,2$, where $U_i\in\{0,1,2\}$ and $0\leq u_i\leq m-1$.
We have
$$
 x_1+x_2=m(U_1+U_2)+(u_1+u_2).
$$
After reduction modulo $3m$, putting $x=x_1+x_2\mmod 3m$ and $u=u_1+u_2\mmod m$, we get
$$
x=\begin{cases}
 m\cdot (U_1+U_2 \mmod 3)+u &\;\text{if $u_1+u_2<m$},\\[.5ex]
 m\cdot (U_1+U_2+1\mmod 3)+u &\; \text{if $u_1+u_2\geq m$}.
\end{cases}
$$

Let us introduce the operation $\oplus_c$, ``addition with carry'' on $\ZZ_m\times\ZZ_3$,
where $\ZZ_m$ and $\ZZ_3$ are treated just as the sets $\{0,1,\dots,m-1\}$ and $\{0,1,2\}$,
respectively:
$$
 (u_1,U_1)\oplus_c(u_2,U_2)=\big(u_1+u_2\mmod m, \;\;(U_1+U_2+\sigma) \mmod 3\big), 
$$
where $\sigma=0$ if $u_1+u_2<m$ and $\sigma=1$ otherwise. 

The situation can be described by the commutative diagram
\begin{equation}
\label{diagram:carry}
 \begin{tikzcd}[column sep=tiny]
&&&[-2em](u_1, U_1) \ar[rr]  
 && x_1
 \\[-5ex]
\text{in $\ZZ_m\times\ZZ_3$}
&& \oplus_c \ar[dd,bend right] 
 &&&&[-1.4em]
 + \ar[dd,bend left] 
 &  \mmod 3m 
 \hspace{2.5em} 
  \\[-5ex]
 &&&
(u_2, U_2) \ar[rr]  
 && x_2
 \\
 &&\phantom{(u, U,\; } 
 \lefteqn{\hspace{-1.3em}(u,U)} 
 &{}\ar[rr]&&{}&
 \phantom{x} \lefteqn{\hspace{-1em}x}
 \end{tikzcd}
\end{equation}

As a matter of fact, the \underline{set}\ $\ZZ_m\times \ZZ_3$ is
still an abelian group under the operation $\oplus_c$ as seen from the right side of the diagram, but
the group operation ``addition modulo $3m$'' is not homomorphic to the canonical addition in the direct sum
of the groups $\ZZ_m\oplus \ZZ_3$.

Similarly, subtraction involves carries in the following way: if $x_i$, $u_i$, $U_i$ are related as above and
if $x=x_1-x_2\mmod 3m$, $u=u_1-u_2\mmod m$, then
$$
x=\begin{cases}
 m\cdot (U_1-U_2\mmod 3)+u &\;\text{if $u_1-u_2\geq 0$},\\[.5ex]
 m\cdot (U_1-U_2-1 \mmod 3)+u &\; \text{if $u_1-u_2<0$}.
\end{cases}
$$

The algorithm of setting up an MSP (Step I of strong starter construction)
mostly follows that described in Sec.~\ref{ssec:3basic}
with modifications reflecting the presence of carries.

\smallskip
Given the triplication table $\Sigma_m$, $m=2q+1$, we first compute two auxiliary tables. 
To every pair $(u_i,v_i)$, $0\le i\le 3q$, of 
$\Sigma_m$ we put in correspondence the
$i$-th {\em difference carry}
$$
\delta_i=\begin{cases}
 0 & \text{if $u_i-v_i\geq 0$},\\
1 & \text{if $u_i-v_i<0$},
\end{cases}
$$
 and the
$i$-th {\em summation carry}
$$
\sigma_i=\begin{cases}
 0 & \text{if $u_i+v_i<m$},\\
1 & \text{if $u_i+v_i\geq m$}.
\end{cases}
$$
Then for $0\le i\le 3q$ the operations $\myminus$ and $\myplus$ are defined by the rules
\begin{equation}
 \label{omop1}
 U_i\myminus V_i= U_i-V_i-\delta_i \mmod 3, \quad  
 \quad U_i \myplus V_i= U_i+V_i +\sigma_i \mmod 3.
 \end{equation}

Hence the difference carries are used to set up the Row Constraints ((1)), while the summation carries are used to set up the Weak Set Constraints ((2)). 
The Color Constraints ((3)) do not involve the operations $\myplus$, $\myminus$
and do not depend on the precomputed carries. The Consistency Constraints ((4)) are vacuous.

Similar to Proposition~\ref{prop:main-mod}, we give an explicit formulation of 
the abstract Theorem~\ref{thm:main} adapted to the Carry scenario. 

\begin{proposition}
\label{prop:main-carry}
Let $\Sigma_m$ be a triplication table. 
Suppose, in Step I of the triplication process, the Sudoku problem is set up using the operations $\myminus$, $\myplus$ given by (\ref{omop1}).
Suppose in Step II a table $\Sigg_{3}$ congruous with  $\Sigma_m$ is found.
If the pairing $S$  is constructed by formulas (\ref{recover-starter}) 
with $F_{\rm Carry}(u,U)$ given by \eqref{carry-recover} and  $(u_i,v_i) \in \Sigma_m$, 
$(U_i,V_i) \in \Sigg_{3}$, $0\le i\le \frac{3(m-1)}2$, then $S$ is a strong starter in $\ZZ_{3m}$. 
\end{proposition}

\begin{example}
\label{ex:carry_m7}
The input data in this example are the same as in the example presented in Sections~3--5 of \cite{OSK25a}.
We take the base starter $T=[(2,3),(4,6),(1,5)]$ of order $7$ and key $\akey=1$, which is admissible.
Here is the triplication table $\Sigma_7=\Sigma_7(T,1)$ found using formulas (\ref{3table-gen}). 
The entries where the difference carries are nontrivial $(\delta_i=1)$
are starred

\bigskip

\centerline{
	$\Sigma_7$=\begin{tabular}{|c||c||c|}
		\hline 
		&$(1, 1)$ &   \\
		\hline
		 $(2, 3)^*$  &  $(3, 4)^*$&  $(5, 6)^*$\\ 
		\hline
		  $(4, 6)^*$ &  $(5, 0)$&  $(2, 4)^*$ \\ 
		\hline
		 $(1,  5)^*$ &  $(2, 6)^*$&  $(3, 0)$ \\ 
		\hline
	\end{tabular}}
	
\bigskip
	
Thus, 
\begin{equation}
\label{del}
\delta_i=\begin{cases}
 0 & \text{for $\;i\in \{0, 5, 9\} $},\\
1 & \text{for $\;i\in \{1,2,3,4,6,7,8\} $}.
\end{cases}
\end{equation}

The weak sets and the weak pair sets (see Definition~\ref{def:wset}) are as follows. 
Starred are those weak pairs for which the summation carries are nontrivial $(\sigma_i=1)$:
$$
\begin{array}{lcl}
W_0=\{2\} & \to & \Wp_0= \{(3,4)^*\}\\
W_3=\{4,9\} & \to & \Wp_3=\{(4,6)^*,\ (3,0)\} \\
W_5=\{1,5\} & \to &\Wp_5=\{(2,3),\ (5,0)\}\\
W_6=\{6,7\} & \to &\Wp_6=\{(2,4),\ (1,5)\}.
\end{array}
$$
Thus, 
\begin{equation}
\label{sig}
\sigma_i=\begin{cases}
 0 & \text{for $\;i\in \{1,5,6,7,9\} $},\\
1 & \text{for $\;i\in \{2,4\} $}.
\end{cases}
\end{equation}

The strong pairs are $(1,1), (5,6), (2,6)$.

Monochrome sets mapped to the subsets of indeterminates, are:
$$
\begin{array}{lclcl}
M_0=\{V_5,V_9\} &&&&\\
M_1=\{U_0, V_0, U_7\} &&
M_3=\{V_1, U_2, U_9\} &&
M_5=\{U_3, U_5, V_7\}\\
M_2=\{U_1, U_6, U_8\} &&
M_4=\{V_2, U_4, V_6\} &&
M_6=\{V_3, V_4, V_8\}
\end{array}
$$

A solution  (one of many) of MSP (see Definition~\ref{def:mod-Sudoku}) is presented  below:

\bigskip

\centerline{
	$\Sigg_3
	(\rm Carry)=$
	\begin{tabular}{|c||c||c|}
		\hline 
		&$(0, 1)$ &   \\
		\hline
		 $(2, 0)$  &  $(2, 2)$&  $(2, 1)$\\ 
		\hline
		  $(1, 2)$ &  $(1, 2)$&  $(1, 0)$ \\ 
		\hline
		 $(2,  0)$ &  $(0, 0)$&  $(1, 1)$ \\ 
		\hline
	\end{tabular}
}

\bigskip	
Let us demonstrate that the constraints ((1)) and ((2)) are satisfied.

Table~\ref{tab:carries_m7} (left) shows the differences $U\myminus V$, see \eqref{omop1}, \eqref{del}. 
We see that, indeed, in each row they are all distinct modulo 3, namely, 1, 2, 0, in that order. 

Table~\ref{tab:carries_m7} (right) shows the sums $U\myplus V$, see \eqref{omop1}, \eqref{sig} for each weak pair. 
The sums modulo 7 of the  weak pairs of  TT are shown in parentheses. We see that (i) the sums $U\myplus V$ modulo 3
 corresponding to equal sums in TT are distinct, and (ii) 
and the sum $U_2\myplus V_2$ corresponding to the weak pair $(3,4)$ from $\Wp_0$ is nonzero.

In addition, the monochrome sets  $M_0=\{1,2\}$  and $M_k=\{0,1,2\}$, $1\le k\le 6$, are as required by the constraints ((3)).

\begin{table}
{\small
	\begin{tabular}{|c||c||c|}
		\hline 
		&$0-1$ &   \\
		\hline
		 $2-0-1$  &  $2-2-1$&  $2-1-1$\\ 
		\hline
		  $1-2-1$ &  $1-2$&  $1-0-1$ \\ 
		\hline
		 $2-0-1$ &  $0-0-1$&  $1-1$ \\ 
		\hline
	\end{tabular}
	\hspace{1em}
	\begin{tabular}{|c||c||c|}
		\hline 
		&strong &   \\
		\hline
		 $2+0\quad(5)$  &  $2+2+1\,(0)$&  strong\\ 
		\hline
		  $1+2+1\;(3)$ &  $1+2\;\quad(5)$&  $1+0 \;(6)$ \\ 
		\hline
		 $2+0\quad(6)$ &  strong&  $1+1 \;(3)$ \\ 
		\hline
	\end{tabular}
	}
\caption{Verification of the constraints ((1)) and ((2)) in Example~\ref{ex:carry_m7}}	
\label{tab:carries_m7}	
\end{table}	
	
The final table  presents the pairs $[(x_i,y_i)]_{i=0}^{9}$ recovered by the formulas
$x_i=7U_i+u_i$, $y_i=7V_i+v_i$ from the tables $\Sigma_7$ and 
$\Sigg_3 (\rm Carry)$ given above.

\bigskip
\centerline{
	\begin{tabular}{|c||c||c|}
		\hline 
		&$(1, 8)$ &   \\
		\hline
		 $(16, 3)$  &  $(17, 18)$&  $(19, 13)$\\ 
		\hline
		  $(11, 20)$ &  $(12, 14)$&  $(9, 4)$ \\ 
		\hline
		 $(15,  5)$ &  $(2, 6)$&  $(10, 7)$ \\ 
		\hline
	\end{tabular}
}

\bigskip

\end{example}

The resulting starter is read off this table:
$S=\{\{1,8\},\{16,3\},\dots, \{10,7\}\}$.
It coincides with that given at the end of \cite{OSK25a} (Appendix, Output). 
The coincidence is explained plainly: we calculated the above solution
of the Carry version of MSP from that starter known beforehand.
As an alternative we could have used an independently found solution of the described MSP.
In general, the method of this section and that of Section~\ref{sec:mod} 
yield the same set of ``triplicated'' starters as follows from Theorem~\ref{thm:scenario-independence}.

\section{The solvability of the Sudoku problem} 
\label{sec:sudoku-solvability}


\subsection{Empirical evidence of solvability}
\label{ssec:sudoku-solv}

As it has already been mentioned, we have no theory as regards solution of the Modular Sudoku Problem  (Step II of triplication process).

In practice, computations for this work as well as for \cite{OSK25a} were done with Python programs based on
the powerful SAT solver \zsat. 

The computational experience has led us to believe that in most cases, given a TT formed according
to Definition~\ref{def:3table}, the MSP
does have a solution and therefore a ``triplicated'' strong starter can be produced.

In this section we illustrate this empirical observation by examples and, to some extent, offer more precise
formulations. 


In Sec.~\ref{ssec:ex-sudoku-special} we consider ``theoretical'' types of triplication tables introduced
in Section~\ref{sec:construcion-tt}. Then, in Sec.~\ref{ssec:ex-sudoku-random} we consider ``random'',
or general,  TTs.

Another object of interest in Sec.~\ref{ssec:ex-sudoku-special} is the set of admissible key values, which has no analog
in the most general case.

\subsection{Sudoku problem for special triplication tables } 
\label{ssec:ex-sudoku-special}

As pointed out earlier, we are able to predict admissible keys only for the construction described in Sec.~\ref{ssec:basic-3}
and assuming the base starter is strong --- see Proposition~\ref{prop:num-good-keys}. 
For other triplication templates considered in Section~\ref{sec:construcion-tt} the number of admissible key values can vary. 

In examples~\ref{ex:ord11-all_disjoint} and~\ref{ex:ord19-keys_table} below
we illustrate this phenomenon for the construction described in Section~\ref{ssec:3base_starters}, using
different selections of three out of four fixed strong starters of orders 11 and 19.
The last example~\ref{ex:pseu} illustrates the situation for a construction described in Section~\ref{ssec:epicycloidal}.

In these and all other computed examples based on templates described in Section~\ref{sec:construcion-tt},
whenever $|K|\ne 0$, a solution of the MSP with $\akey\in K$ did exist.
This observation prompts the following conjecture. 

\begin{conjecture}
\label{conj:sudoku-param1}
Let $\Sigma_m=\Sigma_m(T_0,T_1,T_2,\akey)$,
$\akey\in K(T_0,T_1,T_2)$,  be a triplication table obtained from a triplication template as defined in Definition~\ref{def:key-admissible}.
Then $\Stset{}{\Sigma_m}\neq\emset$. 
\end{conjecture}

Further systematic numerical experiments are needed for a more definitive judgement about validity of this conjecture. 

\begin{example}
\label{ex:ord11-all_disjoint}
We fix four ordered pairwise disjoint strong starters of order 11:
$$
\begin{array}{l}
R_1=[(1, 2), (7, 9), (3, 6), (4, 8), (5, 10)], \\
R_2=[(2, 3), (5, 7), (6, 9), (8, 1), (10, 4)], \\
R_3=[(9, 10), (2, 4), (5, 8), (3, 7), (1, 6)], \\
R_4=[(8, 9), (4, 6), (2, 5), (10, 3), (7, 1)].
\end{array}
$$

For various combinations of three subscripts $i, j, k\in \{1,2,3,4\}$ we have calculated the sets of admissible keys
$K(R_{i},R_{j},R_{k})$. The necessary solvability condition requires that $j\neq k$ and
the construction is symmetric with respect to interchanging the second and the third starters in the triple. 
(Cf.~Remark~\ref{rem:sym}.)
Taking the symmetry into account, any combination can be reduced to one of the following. First,   
\begin{equation}
\label{empty}
K(R_{1},R_{2},R_{3})=K(R_{1},R_{2},R_{4})=K(R_{3},R_{1},R_{4})=K(R_{3},R_{2},R_{4})=
\emset.
\end{equation}
For all remaining cases where $j\ne k$, we obtain $|K(R_{i},R_{j},R_{k})|=5$.

Note that $R_3=R'_1$ and  $R_4=R'_2$. Hence the  combinations $i=j=1, k=3$ and $i=j=2, k=4$ correspond to the  
simplest triplication template defined in Sec.~\ref{ssec:basic-3}.
The result $|K(R_{1},R_{1},R_{3})|=|K(R_{2},R_{2},R_{4})|=5$ is consistent with 
Corollary \ref{cor:keys-case1}(b).
\end{example}

\begin{example}
\label{ex:ord19-keys_table}
In this example we fix four strong starters of order 19:
$$
\begin{array}{l}
S_1=[(15, 16), (4, 6), (10, 13), (8, 12), (2, 7), (14, 1), (17, 5),
(3, 11), (9, 18)], \\
S_2=[(2, 3), (16, 18), (14, 17), (5, 9), (6, 11), (7, 13), (8, 15),
(4, 12), (1, 10)], \\
S_3=[(13, 14), (8, 10), (2, 5), (16, 1), (4, 9), (11, 17), (18, 6),
(7, 15), (3, 12)], \\
S_4=[(11, 12), (4, 6), (17, 1), (9, 13), (3, 8), (10, 16), (14, 2),
(18, 7), (15, 5)].
\end{array}
$$

For various combinations of three subscripts $i,j,k\in\{1,2,3,4\}$ we exhibit in Table~\ref{table:ak} the sets of admissible keys
$K(S_i,S_j,S_k)$. The necessary solvability condition requires that $j\neq k$ and
the construction is symmetric with respect to interchanging columns 1 and 2.
The first half of the table covers the cases where $i=j$; in the second half 
we assume that $i,j,k$ are distinct and $j< k$. 

Admissible keys are marked by crosses.
The last column shows the number of admissible keys, $|K|=|K(S_i,S_j,S_k)|$.

\begin{table}
\centerline{
\setlength{\tabcolsep}{4.5pt}
\small
\noindent
\begin{tabular}{|c|cccccccccccccccccc|c|}
\hline

$i$, $j$, $k$ & 1 & 2 & 3 & 4 & 5 & 6 & 7 & 8 & 9 & 10 & 11 & 12
& 13 & 14 & 15 & 16 & 17 & 18 & $|K|$
\\ \hline

1 1 2 &x&x&&x&x&x&&&&x&x&x&&x&&x&x&& 11 
\\

1 1 3 &x&&&x&x&&x&&x&x&&x&x&x&&x&&& 10 
\\

1 1 4 &&&&&&&&&&&&&&&&&&& 0 
\\

2 2 1 &&x&x&&x&&x&x&x&&&&x&x&x&&x&x& 10 
\\

2 2 3 &x&&x&x&x&x&x&&&x&x&&x&x&&&&x& 11 
\\

2 2 4 &x&x&&x&&x&x&x&x&&x&&&x&&&x&x& 11 
\\

3 3 1 &&&x&&x&x&x&&x&x&&x&&x&x&&&x& 10 
\\

3 3 2 &x&&&&x&x&&x&x&&&x&x&x&x&x&&x& 11 
\\

3 3 4 &&&x&&x&x&&&x&x&x&x&x&x&x&x&x&x& 13 
\\

4 4 1 &&&&&&&&&&&&&&&&&&& 0 
\\

4 4 2 &x&x&&&x&&&x&&x&x&x&x&&x&&x&x& 11 
\\

4 4 3 &x&x&x&x&x&x&x&x&x&x&&&x&x&&x&&& 13 
\\

\hline

1 2 3 &x&&&x&x&&&&&x&&x&&x&&x&&& 7 
\\

1 2 4 &x&x&&&x&x&&&&x&x&&&x&&x&x&& 9 
\\

1 3 4 &x&&&&x&&x&&x&x&&&&x&&x&&& 7 
\\

2 1 3 &&&x&&x&&x&&&&&&x&x&&&&x& 6 
\\

2 1 4 &&&&&&&&&&&&&&&&&&& 0 
\\

2 3 4 &x&&&x&&x&x&&&&x&&&x&&&&x& 7 
\\

3 1 2 &&&&&x&x&&&x&&&x&&x&x&&&x& 7 
\\

3 1 4 &&&&&&&&&&&&&&&&&&& 0 
\\

3 2 4 &&&&&x&x&&&x&&&x&x&x&x&x&&x& 9 
\\

4 1 2 &&x&&&x&&&x&&x&x&x&x&&&&x&x& 9 
\\

4 1 3 &&x&x&x&x&&&x&x&x&&&x&x&&&&& 9 
\\

4 2 3 &x&x&&&x&&&x&&x&&&x&&&&&& 6 
\\
\hline

\end{tabular}
}
\caption{ Admissible keys for TTs constructed from starters in Example~\ref{ex:ord19-keys_table}}
 \label{table:ak}
 \end{table}
\end{example}

The next example deals with triplication tables constructed from epicycloidal pseudostarters (Sec.~\ref{ssec:epicycloidal}).

\begin{example}
\label{ex:pseu}
Consider  the triplication table $\Sigma_7$ found  earlier in Example~\ref{psev-tt}.
One table $\Sigg_3$ congruous with $\Sigma_7$ (uniqueness is not claimed) in the Mod scenario
is presented in Table~\ref{table:7psi} (left). 
The corresponding starter of order 21 recovered by CRT
from $\Sigma_7$ and $\Sigg_3$  can be read from Table~\ref{table:7psi} (right): $S=\{\{10,3\},\{9,17\}, \dots, \{18,7\}\}$.

\begin{table}
\centerline{
	\begin{tabular}{|c||c||c|}
		\hline 
		&$(1, 0)$ &   \\
		\hline
		 $(0, 2)$  &  $(1, 1)$&  $(0, 1)$\\ 
		\hline
		  $(2, 1)$ &  $(2, 2)$&  $(0, 1)$ \\ 
		\hline
		 $(0, 2)$ &  $(2, 2)$&  $(0, 1)$ \\ 
		\hline
	\end{tabular}
\hspace{2em}
	\begin{tabular}{|c||c||c|}
		\hline 
		&$(10, 3)$ &   \\
		\hline
		 $(9, 17)$  &  $(4, 19)$&  $(15, 16)$\\ 
		\hline
		  $(11, 13)$ &  $(5, 14)$&  $(6, 1)$ \\ 
		\hline
		 $(12, 8)$ &  $(20, 2)$&  $(18, 7)$ \\ 
		\hline
	\end{tabular}
}
\caption{A table $\Sigg_3$ congruous with $\Sigma_7$ (Example~\ref{ex:pseu}) and the corresponding
strong starter in $\ZZ_{21}$}
\label{table:7psi}
\end{table}

More  templates with the same starter $T_0$ of Example~\ref{psev-tt} have been studied.
For all available multipliers $(\mu=2,3,4,5)$, we constructed templates with pseudostarters $T_1=\Epic{(7)}(\mu)$ 
and $T_2=T'_1$.
The corresponding admissible keys $\akey\in K(T_0,T_1,T'_1)$  
and sample solutions (ordered strong starters of order 21) are presented in Table~\ref{table:ps}.

\bigskip
\begin{table}
\centerline{
\begin{tabular}{|c| c|}
\hline		 
 $\akey$  &  Sample solution  \\
\hline
\multicolumn{2}{|c|}{$\mu=2, \quad T_1=[(1,2), (2,4), (3,6)]$} \\
\hline
$3$ & $[(10, 3), (9, 17), (4, 19), (15, 16), (11, 13), (5, 14), (6, 1), (12, 8), (20, 2), (18, 7)]$ \\
$5$ & $[(19, 12), (16, 17), (13, 7), (3, 11), (18, 6), (14, 9), (8, 10), (5, 1), (15, 4), (20, 2)]$ \\
$6$ & $[(20, 6), (2, 10), (14, 15), (18, 12), (11, 13), (1, 17), (16, 4), (5, 8), (9, 19), (7, 3)]$ \\
\hline
\multicolumn{2}{|c|}{$\mu=3, \quad T_1=[(4, 5), (1, 3), (5, 1)]$} \\
\hline
$1$ & $[(15, 8), (16, 17), (12, 6), (10, 18), (11, 13), (9, 4), (19, 7), (5, 1), (20, 2), (14, 3)]$ \\
$2$ & $[(16, 2), (9, 3), (20, 7), (18, 19), (4, 13), (10, 5), (6, 8), (12, 1), (14, 17), (15, 11)]$ \\
$4$ & $[(4, 11), (2, 3), (15, 9), (20, 7), (18, 13), (5, 14), (8, 10), (19, 1), (16, 12), (17, 6)]$ \\
\hline
\multicolumn{2}{|c|}{$\mu=4, \quad T_1=[(5, 6), (3, 5), (1, 4)]$} \\
\hline
$3$ & $[(3, 17), (9, 10), (15, 2), (11, 5), (4, 20), (6, 8), (19, 7), (12, 1), (18, 14), (13, 16)]$ \\
$5$ & $[(5, 19), (9, 17), (10, 4), (6, 7), (18, 13), (15, 3), (14, 16), (12, 8), (20, 2), (1, 11)]$ \\
$6$ & $[(6, 13), (9, 3), (11, 12), (7, 15), (18, 20), (16, 4), (1, 17), (19, 8), (14, 10), (2, 5)]$ \\
\hline
\multicolumn{2}{|c|}{$\mu=5, \quad T_1=[(2, 3), (4, 6), (6, 2)]$} \\
\hline
$1$ & $[(1, 8), (9, 3), (17, 18), (12, 20), (4, 6), (5, 14), (16, 11), (19, 15), (7, 10), (13, 2)]$ \\
$2$ & $[(16, 2), (9, 10), (11, 19), (13, 7), (18, 6), (20, 15), (3, 5), (12, 1), (8, 4), (14, 17)]$ \\
$4$ & $[(4, 18), (16, 17), (20, 14), (1, 9), (11, 13), (15, 3), (12, 7), (19, 8), (10, 6), (2, 5)]$ \\
\hline
\end{tabular}
}
\caption{Sample strong starters in $\ZZ_{21}$ congruous with $\Sigma_7(T_0, T_1,T'_1,t)$
for $T_1=\Epic{(7)}(\mu)$ and all admissible keys}
 \label{table:ps}
 \end{table}
\end{example}

\begin{remark}
In Example~\ref{ex:pseu}, the number of admissible keys was always equal to $(m-1)/2$. Our numerical experiments 
show that it is often the case when $m$ is prime, $T_0$ is a strong starter,
 and $T_2=T'_1$ are epicycloidal pseudostarters. 
However, there are exceptions. For example, if $m=13$, $T_0=[(3, 4), (6, 8), (9, 12), (10, 1), (2, 7), (5, 11)]$, and $\mu=3$, there are only 3 admissible keys: $K=\{4,10,12\}$. 
\end{remark}

\subsection{Sudoku problem for a general triplication table} 
\label{ssec:ex-sudoku-random}

In this section we discuss triplication tables whose columns are formed by balanced triples of pseudostarters (Definition~\ref{def:3-balanced}) but do not fall under any of the explicit constructions 
described in Section~\ref{sec:construcion-tt}. 

\begin{example}
\label{ex:wild-table}

The starter is the same as in Example~7.7 of \cite{OSK25a},
$$
S=\{\{13, 12\}, \{19, 17\}, \{7, 4\}, \{10, 14\}, \{15, 20\}, \{3, 9\}, \{1, 8\}, \{5, 18\}, \{11, 2\}, \{16, 6\}\}.
$$
In \cite{OSK25a} we showed that it cannot be obtained by triplication from any table $\Sigma_7(T,t)$,
which in the present paper is defined in Sec.~\ref{ssec:basic-3}.

Let us prove a stronger fact: the reduction of $S$ modulo 7 cannot be arranged in a table $\Sigma_7$ in accordance with rules of Section~\ref{sec:construcion-tt} (Definition~\ref{def:3-template}). 

It is already noted in \cite{OSK25a} that the key value must be $\akey=1$, coming from the pair $(1,8)$.
The rows of the regular part of a hypothetical $\Sigma_7$ must contain pairs as follows:
$$
\ba{ll c l}
\text{for difference $1$:} & (5,6) & (2,3) & (4,5)\\ 
\text{for difference $2$:} & (3,5) &(2,4) &(6,1) \\
\text{for difference $3$:} & (6,2) &(4,0) &(0,3) 
\ea
$$
(We have identified the composition of the rows but we do not know in which column every pair belongs.)

The pair $(6,1)$ in the 2nd row must be in column 0, since otherwise column 0 would not contain 1.
Also, the pair $(6,2)$ in the 3rd row must be in column 0, since the other two pairs of that row contain 0. 
We see that there will be at least two $6$s in column 0 of $\Sigma_7$; so there can be no partition (and therefore no starter) in that column. 

Nevertheless, the nature of our Definition~\ref{def:3table-s}  ensures that
a triplication table induced by $S$ exists. 
One such table $\Sigma^S_7$ is displayed as Table~\ref{table:7-3-wild} (left).
(Other possible tables from the same equivalence class can be obtained by arbitrary permutations of pairs
within rows, cf.~Remark~\ref{rem:non-unique-sigma-m}.)

Table~\ref{table:7-3-wild} (middle) displays the appropriately arranged table $\Sigg^S_3$ in Mod scenario, 
Table~\ref{table:7-3-wild} (right) is the original starter interpreted as the pairing congruous with $\Sigma^S_7$ via $\Sigg^S_3$
(cf.~Definition~\ref{def:mod-Sudoku-starter}).
 
Another solution of the same MSP is presented in Table~\ref{table:7-3-wild2} (middle).
It is a table $\Sigg_3$ congruous to $\Sigma_7^S$ but different from $\Sigg^S_3$. 
Table~\ref{table:7-3-wild2} (right) gives a strong starter of order 21 congruous with $\Sigma_7^S$ via this $\Sigg_3$ and different from $S$.

\begin{table}
 
\noindent
\begin{tabular}{|c|c|c|}
\hline
&(1,1)&\\
\hline
(5,6)&(2,3)&(4,5)\\
\hline
(3,5)&(2,4)&(6,1)\\
\hline
(6,2)&(4,0)&(0,3)\\
\hline
\end{tabular}
\hspace{0.1em}
\begin{tabular}{|c|c|c|}
\hline
&(1,2)&\\
\hline
(0,1)&(0,0)&(0,2)\\
\hline
(2,1)&(2,2)&(2,0)\\
\hline
(0,1)&(1,1)&(2,1)\\
\hline
\end{tabular}
\hspace{0.1em}
\begin{tabular}{|c|c|c|}
\hline
&(1,8)&\\
\hline
(12,13)&(9,3)&(18,5)\\
\hline
(17,19)&(2,11)&(20,15)\\
\hline
(6,16)&(4,7)&(14,10)\\
\hline
\end{tabular}
\caption{The tables $\Sigma^S_7$, $\Sigg^S_3$, and the given starter $S$ of Example~\ref{ex:wild-table}} 
\label{table:7-3-wild}
\end{table}

\begin{table}

\noindent
\begin{tabular}{|c|c|c|}
\hline
&(1,1)&\\
\hline
(5,6)&(2,3)&(4,5)\\
\hline
(3,5)&(2,4)&(6,1)\\
\hline
(6,2)&(4,0)&(0,3)\\
\hline
\end{tabular}
\hspace{0.1em}
\begin{tabular}{|c|c|c|}
\hline
&(2,0)&\\
\hline
(0,0)&(2,1)&(0,1)\\
\hline
(0,2)&(0,1)&(1,1)\\
\hline
(2,1)&(2,2)&(1,2)\\
\hline
\end{tabular}
\hspace{0.1em}
\begin{tabular}{|c|c|c|}
\hline
&(8,15)&\\
\hline
(12,6)&(2,10)&(18,19)\\
\hline
(3,5)&(9,4)&(13,1)\\
\hline
(20,16)&(11,14)&(7,17)\\
\hline
\end{tabular}

\caption{Another Sudoku solution from the same triplication table $\Sigma^S_7$}
\label{table:7-3-wild2}
\end{table}
\end{example}

Finally, we demonstrate that an extension of Conjecture~\ref{conj:sudoku-param1} to the set of all possible triplication tables 
is false. 

\begin{example}
\label{ex:counterex-conj-solv}
Two examples of TTs with $m=11$ and $m=13$ 
for which the Modular Sudoku Problem does not have a solution are shown in Table~\ref{table:tt-nosol}.
These  examples were found using 
randomized numerical construction of general triplication tables by \zsat-based program where constraints specified
by Definition~\ref{def:3table} were implemented.
 \end{example}

\begin{table}[ht]
 \centerline{
\begin{tabular}{|c|c|c|}
\hline
&(7,7)&\\
\hline
(1,2)&(10,0)&(7,8)\\
\hline
(4,6)&(3,5)&(1,3)\\
\hline
(6,9)&(1,4)&(10,2)\\
\hline
(6,10)&(5,9)&(4,8)\\
\hline
(0,5)&(8,2)&(9,3)\\
\hline
\end{tabular}
\hspace{3em}
\begin{tabular}{|c|c|c|}
\hline
&(10,10)&\\
\hline
(3,4)&(10,11)&(4,5)\\
\hline
(9,11)&(3,5)&(12,1)\\
\hline
(1,4)&(8,11)&(5,8)\\
\hline
(2,6)&(9,0)&(3,7)\\
\hline
(2,7)&(9,1)&(7,12)\\
\hline
(2,8)&(6,12)&(0,6)\\
\hline
\end{tabular}
}
\caption{Triplication tables which do not yield a solvable Sudoku problem }
\label{table:tt-nosol}
\end{table}

\begin{remark}
While a generalization of Conjecture~\ref{conj:sudoku-param1} fails to be unconditionally true,
our numerical experiments show that ``problematic" TTs of general structure are rare. 
The great majority of randomly generated TTs do yield a strong starter. 
Table~\ref{table:counterex-gen} shows the number $N_{\emset}$ of TTs for which the MSP does not have a solution vs the total number $N$ of generated TTs for odd orders 5 to 19.

\begin{table}[ht]
\centerline{
\begin{tabular}{|c|c|c|c|c|c|c|c|c|}
\hline
$m$ & 5 & 7 & 9 & 11 & 13 & 15 & 17 & 19 \\
\hline
$N$ &  4000 & 4000 & 1000 & 3000 & 9000 & 1500 & 72000 & 45000\\
\hline
$N_{\emset}$ &  0 & 0 & 27 & 2 & 1 &  1 & 0 & 0 \\
\hline
\end{tabular}
}
\caption{The number of ``bad'' triplication tables in the generated samplings}
\label{table:counterex-gen}
\end{table}

The question of how one can take advantage of the flexibility offered by our definition of TT 
to ensure the solvability of MSP, requires further investigation. 
\end{remark}

\section{Conclusion}
\label{sec:conc}
In this paper we develop an approach for constructing strong starters in $\ZZ_{3m}$, where $m$ is an odd integer. This approach extends our previous method  \cite{OSK25a}, which was applicable only for $m$ coprime with 3.  
Here we broaden the concepts of the triplication table and the Modular Sudoku Problem so
that any strong starter of order $3m$ can potentially be found as a result of triplication.
New types of theoretical constructions of triplication tables are described. 
In  \cite{OSK25a} we employed the SAT solver \zsat\ for solving the Sudoku problem. 
Here we employ it also for the purpose of construction of a general, ``random'' triplication table. 
Notwithstanding the practical success of the triplication method, the existence of triplication tables
for any odd $m>3$ for which the Modular Sudoku Problem has a nonempty solution set remains unproven.

\section*{Appendix 1: Sets, multisets, tuples, datasets} 
\label{app:data}

\renewcommand{\thesection}{A1}

In plain language, sufficient for our purposes, a {\em set}\ is a
collection of elements without repetitions.

A multiset is a collection of elements possibly with repetitions.
Rigorously, a multiset is defined as a {\em multiplicity function}\ on 
the {\em set}\ of distinct elements, called {\em support}\ of the multiset, with values in $\NN$. 

For example, a multiset $\{1,5,1,2,4,2,2\}$ can be written as
$A=\{1^2,\, 2^3,\, 4^1,\,5^1\}$, where
the entries $x^{\nu}$ represent the function $\nu$ from the set
$\supp A=\{1,2,4,5\}$ to $\NN$.

The {\em union of multisets} $A$ and $B$ is the multiset
$A\uplus B$ with support $C=\supp A\cup \supp B$ and the multiplicity function
$\nu_C(x)=\nu_A(x)+\nu_B(x)$, where $\nu_A$ and $\nu_B$ are the multiplicity functions
of the multisets-summands extended by 0 from their natural domains to all of $C$.
For instance, $\{1^2, 2^1\}\uplus\{1^1, 3^2\}=\{1^3, 2^1, 3^2\}$.

A {\em tuple}, or {\em array}, is a linearly ordered collection
(indexed, if needed, by consecutive integers starting from 0 or 1 as
specified in particular situations). 

We use  braces $\{\,\}$ as delimiters for sets or multisets when they are given by lists of their elements,
and brackets $[\,]$  for tuples. In the special case of ordered pairs (2-element tuples) 
we use parentheses $(\,)$ as delimiters.

Thus a 3-tuple of ordered pairs may look like
$[(2,3), (0,2), (5,3)]$.

We use the term {\em pair} to mean either ordered (tuple) or unordered (set, multiset) pair,
distinguished by the delimiters used.
Note that $\{a,b\}=\{b,a\}$, but $(a,b)\neq (b,a)$ unless $a=b$.

A {\em dataset}\ is, generally, a recursively organized arrangement of tuples, that is, a tuple whose entries are either
``atomic'' elements (numbers) or datasets of smaller depth. 

Atomic elements in a dataset are referred to as {\em components}. 
Every component has a uniquely determined position (placeholder). A function that transforms atomic elements
to atomic elements naturally induces a {\em componentwise transformation}\ of datasets preserving their structure.

We do not pedantically define datasets in full generality. 
Examples will suffice.

A tuple of (ordered) pairs $T=[(1,15),(7,12),(4,10),(1,3)]$ is a dataset.  
Its bracketing structure is $[(\,),(\,),(\,),(\,)]$ and the components are $1,15,7,12,4,10,1,3$. The place where each component is situated in the hierarchical structure defined by bracketing is important.
For instance, the datasets $[(7,12),(1,15),(4,10),(1,3)]$ or $[(15,1),(7,12),(4,10),(1,3)]$ or $[1,15,7,12,4,10,1,3]$
are all distinct and different from $T$. The former two have the same bracketing structure as $T$, while the last one has
different bracketing structure. All four datasets (including $T$) have the same multiset of components. 

The term {\em pairing}\ in this paper can mean a (multi)set whose elements are unordered pairs
or a dataset which is a tuple of ordered pairs of numbers. The appropriate delimiters exclude possibility of confusion.

\section*{Appendix 2: Recovering a number from residues in the case of non-coprime moduli}
\label{app:genGCD}

\renewcommand{\thesection}{A2}

The number-theoretic problem \eqref{congmod3m}--\eqref{compat} is crucial for the variant of triplication 
considered in Section~\ref{sec:9etc}. 
Let us discuss its solvability and a method of solution.
Sufficiency of the necessary condition \eqref{compat} will follow from the solution algorithm
presented below, cf.~\cite[Sec.~3.5]{Jones_98}.

\smallskip
The method to find $x$ from the congruences \eqref{congmod3m} is as follows.

\smallskip
1. In the particular case where $u\equiv U\equiv 0\mmod 3^{\nu}$ one can divide the congruences by $3^\nu$
to get
$$
 x=3^{\nu} x', \qquad  x'\equiv u/3^\nu \mmod p,\quad  x'\equiv U/3^\nu \mmod 3,
$$
and $ x'$ is recovered by CRT, since $\gcd(p,3)=1$.

\smallskip
2. In the general case, let $\bar u$ be the common value of $u\mmod 3^\nu$ and $U\mmod 3^\nu$.
Then $\bar x=x-\bar u$ satisfies the congruences 
$$\bar x \,\equiv\, u-\bar u \!\pmod{m}, 
\qquad
\bar x\,\equiv\, U-\bar u \!\pmod{3^{\nu+1}},
$$
which go under the particular case described in item 1. Therefore, $x=\bar u+\bar x$ and  formulas (\ref{CRT-gen-1}),  (\ref{CRT-gen-2}) follow.

\smallskip
For reference, let us present a formal statement. 

\begin{lemma}
\label{lem:CRTgen}
Let $m,h$ be positive integers with $\gcd(m,h)=d$, $\lcm (m,h)=n$.
Given $u\in\{0,\dots,m-1\}$ and $U\in\{0,\dots,h-1\}$ such that $u\equiv U \!\!\pmod{d}$,
there is a unique $x\in\{0,\dots,n-1\}$ satisfying the congruences $x\equiv u \!\!\pmod{m}$, $x\equiv U \!\!\pmod{h}$.
\end{lemma}

In our context, the parameters in Lemma are:  $m=3^\nu p$, $h=3^{\nu+1}$, $n=3m$,  $d=3^\nu$.

\begin{example}
Let $m=45=3^2\cdot 5$, so $p=5$, $\nu=2$. Then $h=3^{\nu+1}=27$, $n=3m=135$, $d=3^\nu=9$.

Consider the problem (\ref{congmod3m}) with $u=22$ and $U=13$: find $x\mmod 135$ such that 
$$
x\equiv 22 \!\pmod{45}\quad  \text{and} \quad x\equiv 13 \!\pmod{27}.
$$
The compatibility condition (\ref{compat}) holds, 
since 
$$
22\equiv 13 \equiv 4 \!\!\pmod{9},
$$
 and $\bar u=4$.  Thus, by (\ref{CRT-gen-1}) we have 
 $$
 x'\equiv 2\!\!\pmod{5},\qquad x'\equiv 1\!\!\pmod{3}.
$$
 Using the Extended Euclid's Algorithm or just by observation, we find the answer: $x'=7 \mmod{15}$.
Now by (\ref{CRT-gen-2}), $x=(4+7\cdot 9)\mmod{135}$, so $x=67\mmod{135}$.
Answer check: $67\equiv 22\!\pmod{45}$ and $67\equiv 13\!\pmod{27}$. 
\end{example}

\section*{Acknowledgements}
M.K. would like to thank ACENET 
and the Digital Research Alliance of Canada 
for enabling computations reported in Section~\ref{ssec:ex-sudoku-random} of this paper.

\end{document}